\documentclass[11pt]{amsart}
\usepackage{geometry}
\usepackage{fullpage}
\usepackage[parfill]{parskip}    % Activate to begin paragraphs with an empty line rather than an indent
\usepackage{graphicx}
\usepackage{amssymb}
\usepackage{amsmath}
\usepackage{epstopdf}
\usepackage{amsmath}
\usepackage{thm-restate}
\addcontentsline{toc}{appendix}{APPENDICES}

\usepackage{amsmath,amsthm,amscd,amssymb}
\usepackage{latexsym}
\usepackage[colorlinks,citecolor=red,pagebackref,hypertexnames=false]{hyperref}
\numberwithin{equation}{section}

\theoremstyle{plain}
\newtheorem{theorem}{Theorem}[section]

\newtheorem{lemma}[theorem]{Lemma}
\newtheorem{corollary}[theorem]{Corollary}
\newtheorem{proposition}[theorem]{Proposition}
\newtheorem{conjecture}[theorem]{Conjecture}

\theoremstyle{definition}
\newtheorem{definition}[theorem]{Definition}

\newtheorem{problem}[theorem]{Problem}
\newtheorem{example}[theorem]{Example}

\theoremstyle{remark}
\newtheorem{remark}[theorem]{Remark}

\newtheorem{case[theorem]}{Case}

  \newtheorem*{Dutkay-Lai}{{\bf Theorem A}}
  \newtheorem*{Fuglede-Gabor-Problem}{{\bf Fuglede-Gabor Problem}}

\def\norm#1.#2.{\lVert#1\rVert_{#2}}

\begin{document}

\title{Non-separable lattices, Gabor orthonormal bases and Tilings}

\author{Chun-Kit Lai and Azita Mayeli}

\date{\today}

\address{Chun-Kit Lai, Department of Mathematics, College of Science and Engineering,  San Francisco State Univesrity, San Francisco, USA} 
\email{cklai@sfsu.edu}
\address{Azita Mayeli, Department of Mathematics and Computer Science, Queensborough and The Graduate Center, City University of New York, USA} 
\email{AMayeli@gc.cuny.edu}

\thanks{The work of the second author was partially supported by the PSC-CUNY Grant  B  cycles 47 and 48.}

  \maketitle 
  
\begin{abstract} Let  $K\subset \Bbb R^d$ be a bounded set  with positive Lebesgue measure. Let $\Lambda=M(\Bbb Z^{2d})$ be a lattice in $\Bbb R^{2d}$ with density dens$(\Lambda)=1$. It is well-known that if $M$ is a diagonal block matrix with diagonal matrices $A$ and $B$,    then
$\mathcal G(|K|^{-1/2}\chi_K, \Lambda)$ is an orthonormal basis for $L^2(\Bbb R^d)$ if and only if $K$ tiles both by  $A(\Bbb Z^d)$ and $B^{-t}(\Bbb Z^d)$. However, there has not been any intensive study when $M$ is not a diagonal matrix. We investigate this problem for a large  class of  important  cases
 of $M$. In particular, if $M$  is any lower block triangular matrix with diagonal matrices $A$ and $B$, we prove that  if $\mathcal G(|K|^{-1/2}\chi_K, \Lambda)$ is an orthonormal basis,   then $K$ can be written as a finite union of fundamental domains of $A({\mathbb Z}^d)$ and at the same time, as a finite union of fundamental domains of $B^{-t}({\mathbb Z}^d)$. If $A^tB$ is an integer matrix, then there is only one common fundamental domain, which means $K$ tiles by a lattice and is spectral. However, surprisingly, we will also illustrate by an example that a union of more than one fundamental domain  is also possible. We also provide a constructive way for forming a Gabor window function for a given upper triangular lattice.
  Our study is related to a Fuglede's type problem in Gabor setting and we give a partial answer to this problem in the case of lattices.
 % Our results  give a partial answer to  the  long standing open question that,  which functions $g$ serves  as a  window function for a Gabor orthonormal basis for $L^2(\Bbb R^d)$,  and solve Liu and Wang\rq{}s Conjecture (\cite{LW03}) for special cases of non-separable lattices with Gabor window  function of the form  $g(x)=|K|^{-1/2}\chi_K$.
\end{abstract}

{\bf Keywords:} Non-separable lattices, Gabor orthonormal bases, tiling and spectral sets.

\section{Introduction}

% \begin{definition}[Gabor systems]\label{Gabor system}
Let $g\in L^2(\Bbb R^d)$,  and let $\Lambda\subset \Bbb R^{2d}$ be  a countable subset. We define the {\it Gabor system} (also known as {\it Weyl-Heisenberg system}) $\mathcal G(g, \Lambda)$ with respect to $g$ and $\Lambda$  to be  the  collection of functions $\pi(a, b)g$ defined by  combined time and frequency shifts of $g$:
$$
\pi(a,b)g(x)= M_{b}T_{a}g =  e^{2\pi i \langle b, x\rangle} g(x-a) \quad (a,b)\in \Lambda .
$$
$\Lambda$ is also known as  {\it time-frequency set}  and
the frequency shift is also called modulation. We say $g$ is an {\it orthonormal Gabor window function}  with respect to $\Lambda$, or simply a window function, if  $\mathcal G(g, \Lambda)$ is an orthonormal basis for $L^2(\Bbb R^d)$. See e.g. 
\cite{OleBook} or \cite{Groechenig-book}.

%\end{definition}

We  call $\Lambda$ {\it separable} if it is of the form of $\Lambda = {\mathcal T}\times \Gamma$ for some countable subsets ${\mathcal T}$ and $\Gamma$ in ${\mathbb R}^d$. Gabor systems  have been introduced for the first time in 1946 by Gabor \cite{Gabor46} and are now  fundamental objects in applied and computational harmonic analysis. Moreover, for $\mathcal G(g, \Lambda)$ to be a Gabor orthonormal basis,  the (Beurling) density of $\Lambda$, denoted by dens$(\Lambda)$, must be $1$ \cite{RS}.

\medskip

The existence of a window function for a given lattice  has been investigated  for several  special cases of $M$. The question of existence has been  completely answered
by Han  and Wang \cite{HanW1}  for separable lattices of the form  $\Lambda= \mathcal J\times \mathcal T$  with  dens$(\Lambda)=1$. They answered the question by showing the existence of a common fundamental domain for two different lattices. Later, the same authors  partially answered the question for non-separable lattices (i.e. the lattices of not of the form of $\mathcal J \times \mathcal T$)  for special cases of matrix $M$ \cite{HanW4}. Indeed, they proved that,  when  for example $M$ is a block triangular matrix, a  window function $g$ exists  and it can be chosen so  that $|g|$ or $|\hat g|$ is the scalar multiple of a characteristic function. They also showed the existence of a window function with compact support for rational matrices  $M$.

\medskip

Given a subset $K\subset \Bbb R^d$,
we denote by $\chi_K$ the indicator function of $K$ and by $|K|$ its Lebesgue measure.
%
%In this paper we study the question that,  given a full-rank lattice $\Lambda = M({\mathbb Z}^{2d})\subset \Bbb R^{2d}$ with dens$(\Lambda)= |\det(M)|^{-1} =1$ and a set $K\subset \Bbb R^d$ of finite positive measure, what is the  necessary   conditions for  $K$ such that   $\mathcal G(|K|^{-1/2} \chi_K, \Lambda)$ is   a Gabor orthonormal basis for $L^2(\Bbb R^d)$?
The main focus of this paper  is  the following. Suppose that $\Lambda = M({\mathbb Z}^{2d})\subset \Bbb R^{2d}$ is a full lattice with Beurling density dens$(\Lambda)= |\det(M)|^{-1} =1$   and suppose that $\mathcal G(|K|^{-1/2} \chi_K, \Lambda)$ forms a Gabor orthonormal basis for $L^2(\Bbb R^d)$. What can we say about the structure of $K$? This question is related to the study of spectral sets and translational tiles which we will call the {\it Fuglede-Gabor Problem} later on.

\medskip

\begin{definition}[Spectral and tiling sets]\label{Spectral and tiling sets}   A Lebesgue measurable set
  $K\subset \Bbb R^d$  with positive  and finite measure is a  {\it spectral set} in $\Bbb R^d$ if there is a countable set $B\subset \Bbb R^d$ (not necessarily unique) such that
  exponentials $\{e_b(x):= e^{2\pi i \langle b,x\rangle} : b\in B, x\in K\}$ constitute an orthogonal basis for   $L^2(K)$, i.e., the exponentials are mutual orthogonal and complete in $L^2(K)$.  In this case $B$ is called a {\it spectrum} for $K$.

  \medskip

 We say $K$ {\it multi-tiles} $\Bbb R^d$ by its translations  if there
 is a countable set ${\mathcal J}\subset \Bbb R^d$ and integer $N\ge 1$ such that
 % such that  $|(K+t)\cap (K+t')|=0$, for all distinct $t, t\rq{}\in \mathcal J$,   and $\bigcup_{t\in {\mathcal J}} K+t$ covers $\Bbb R^d$.  It is easy to check that the tiling condition  is equivalent to say that
\begin{equation}\label{multitile-generic}
\sum_{t\in{\mathcal J}} \chi_{K} (x+t)=N\quad a.e. \ x\in \Bbb R^d .
\end{equation}
If $N=1$, then $K$ {\it tiles} ${\mathbb R}^d$ and  the set ${\mathcal J}$ is called {\it tiling set} for $K$ (For more details about multi-tiles, see e.g. \cite{Kol}).
 \end{definition}

\medskip

  Spectral sets   have been studied extensively in the recent years and their study has been reduced to the study of  tiling sets by the Fuglede Conjecture or Spectral Set Conjecture \cite{Fug74}   which asserts: {\it A set $K\subset \Bbb R^d$ with positive and finite measure is a spectral set if and only if $K$ tiles $\Bbb R^d$ by translations. }
 Fuglede proved the conjecture in his celebrated 1974 paper \cite{Fug74} for the case  when $K$ tiles by a lattice or $K$ has a spectrum which is a lattice. The Fuglede Conjecture led to considerable activity in the past three decades. In
  2004, Tao \cite{T04} disproved the Fuglede conjecture for dimension $5$ and higher, followed by Kolountzakis and Matolcsi\rq{}s result \cite{FMM,KM06}  where they proved  that  the conjecture fails in dimensions $3$ and higher. For more recent results and historical comments see e.g. \cite{BHM16,IMP17}.

\medskip
Spectral sets and tiles appear naturally in the Gabor setting.  Indeed, let $\Lambda= \mathcal J \times \mathcal T$ be a separable countable set (not necessarily a lattice) with  dens$(\Lambda)=1$ and let  $\Omega$ be a compact set  in $\Bbb R^d$  which tiles by $\mathcal J$ and is spectral for $\mathcal T$.  For example, take $\Omega = [0,1]^2$ and ${\mathcal J} = {\mathcal T} = {\mathbb Z}^2$. Let $g$ be a   function supported in $\Omega$.  Then an easy calculation shows that the Gabor system
 $\mathcal G(g, \Lambda)$ is an orthonormal basis  if   $|g(x)|=|\Omega|^{-1/2} \chi_\Omega(x)$
 (see also \cite{LW03}, Lemma 3.1).  We call such Gabor bases {\it standard}.  Liu and Wang \cite{LW03} conjectured the converse of this result that  for  a  compactly supported function $g$ and  a countable separable set $\Lambda= \mathcal J \times \mathcal T$, if $\mathcal G(g, \Lambda)$  is an orthonormal basis for $L^2(\Bbb R^d)$, then there is a compact set $\Omega\subset \Bbb R^d$ such that
 $|g|$ is a constant multiple of $\chi_\Omega$,    $\Omega$ tiles by $\mathcal J$  and is a spectral set for $\mathcal T$. Liu and Wang proved their conjecture when the support of $g$ is an interval.
  Dutkay and the first listed author recently proved that the Liu and Wang\rq{}s conjecture is affirmative if $g$ is non-negative \cite[Theorem 1.8]{Dutkay-Lai}.  % Indeed, they proved that for $g$ is a non-negative function with bounded and measurable support $K$ in $\Bbb R^d$,    $\Lambda=
 %\mathcal J\times \mathcal T$ a countable set with $D(\Lambda)=1$,  if
 %$\mathcal G(g, \Lambda)$ is an orthonormal basis for $L^2(\Bbb R^d)$, then
  %$g=  |K|^{-1/2}\chi_K$ a.e. on $K$ and $K$ tiles by $\mathcal J$ and is a spectral  with respect to $\mathcal T$.
 %(\cite{Dutkay-Lai}, Theorem 1.8 resp. Theorem 6.2)
However, the conjecture  is still unsolved for general compactly supported $g$.

\iffalse
\begin{Dutkay-Lai}\label{Dutkay-Lai}(\cite{Dutkay-Lai}, Theorem 1.8 resp. Theorem 6.2) Suppose that $g$ is a non-negative function with bounded and measurable support $K$ in $\Bbb R^d$. Suppose that    $\Lambda=
 \mathcal J\times \mathcal T$ is a countable set with $D(\Lambda)=1$. If  $\mathcal G(g, \Lambda)$ is an orthonormal basis for $L^2(\Bbb R^d)$, then the following hold.
\begin{itemize}
\item[(i)] $K$ tiles $\Bbb R^d$ by translations by $\mathcal J$.
\item[(ii)] $g=  |K|^{-1/2}\chi_K$ a.e. on $K$.
\item[(iii)] $K$ is a spectral with respect to $\mathcal T$.
\end{itemize}
\end{Dutkay-Lai}

We shall sketch the proof of Theorem A here. The completeness of the Gabor system $\mathcal G(g, \Lambda)$ proves that $\cup_{j\in \mathcal J} K+j$ is a cover for $\Bbb R^d$. The orthogonality implies that the collection of sets $\{K+j: \ j\in \mathcal J\}$ are mutual disjoint up to a Lebesgue measure zero set, thus  $K$ is a tiling set with respect to $\mathcal J$. Further, the measure  $d\mu(x):=g(x)^2 dx$ has a spectrum by proving an equivalent statement that
%
$$\sum_{p\in \mathcal T} |\widehat{\mu} (x-p)|^2 =1.$$
%
This implies that the measure $\mu$  is spectral, thus   $|g|^2$ must be a constant multiple of a characteristic function by
Corollary 1.4 of \cite{Dutkay-Lai}. This forces  $g= |K|^{-1/2}\chi_K$ a.e. on its support $K$, hence  the   conclusions of (ii) and (iii).  The converse holds by a direct calculation.
%The converse holds by Lemma 3.1 in \cite{LW03}.
\fi

The following problem links the study of window functions associated with Gabor orthonormal bases to the tiling and spectral properties of sets.
\medskip

\begin{problem}\label{our conjecture1}(Fuglede-Gabor Problem)
Let    $K\subset \Bbb  R^d$ be a measurable subset  with positive and finite measure, and let   $\Lambda \subset \Bbb R^{2d}$ be a countable subset. If the  Gabor family $\mathcal G(|K|^{-1/2}\chi_K, \Lambda)$ is an orthonormal basis for $L^2(\Bbb R^d)$, then $K$ tiles and is a spectral set.
 \end{problem}

Since the indicator function of a set is non-negative, {\it Problem \ref{our conjecture1} is already affirmative if the time-frequency set is a separable countable sets} using the results of Dutkay and the first listed author.
 Therefore we only focus on the case when the time-frequency set is non-separable. Moreover, although the problem does not require any boundedness assumption of $K$, our interest will mainly be focused on the set $K$ being bounded.

   It is hard to speculate whether Problem \ref{our conjecture1}  is true or not in its full generality. But from   the point of view of Fuglede's result for lattices, we still hope that   the Fuglede-Gabor problem is true for  non-separable lattices  as well. Unfortunately, after our intensive study, we found out that, similar to many notoriously difficult problems in Gabor analysis (see e.g. \cite{Groechenig-mystery}),   % (e.g. HRT conjecture and $abc$ problem),
    the Fuglede-Gabor problem for lattices appears to be uneasy. This paper     gives a partial answer towards the full solution together with some unexpected examples, as we explain below.  %Unfortunately, there is %  and in this paper we study
%  some important cases of $\Lambda$. Our results also solve  Liu and Wang\rq{}s Conjecture for a large class of non-separable lattices $\Lambda$  under the assumption that the window function is in the form of  $g(x)=|K|^{-1/2}\chi_K$.

\noindent{\bf Main Results of the paper.} Our main results will mostly be focused on the lower triangular block matrices since  most    matrices can be reduced to the lower triangular form:
   \begin{equation}\label{lower_block}
   \Lambda = \left(
                                      \begin{array}{cc}
                                        A & O \\
                                        C & B \\
                                      \end{array}
                                    \right)({\mathbb Z}^{2d}), \ \mbox{and} \ |\det(AB)|=1
                                    \end{equation}
    (i.e. dens$(\Lambda)=1$). We will use $B^{-t}$ to denote the inverse  transpose of  the matrix $B$.  Our first general key lemma is as follows, it will serve as a key step for our further analysis.

    \medskip

    \begin{lemma}[Key Lemma]\label{Th_union of FD} Let  $\Lambda = M({\mathbb Z}^{2d})$ with $M$ an  $2d\times 2d$ invertible  lower triangular block matrix of the form (\ref{lower_block}).
Suppose that ${\mathcal G}\left({|K|^{-1/2}}\chi_K,\Lambda\right)$ is a Gabor orthonormal basis. Then there exists an integer $N\ge 1$ such that
                                    $$
                                    K = \bigcup_{j=1}^N D_j=\bigcup_{j=1}^N E_j
                                    $$
where $D_j$\rq{}s are fundamental domains of $B^{-t}(\Bbb Z^d)$ and $E_j$\rq{}s are almost disjoint fundamental domains of $A(\Bbb Z^d)$ with $|D_i\cap D_j| = 0$ and $|E_i\cap E_j| =0$ for all $i\ne j$. (i.e. K  multi-tiles  ${\mathbb R}^d$ simultaneously by $A({\mathbb Z}^d)$ and $B^{-t}(\Bbb Z^d)$.)
\end{lemma}

 If we can prove that $N=1$, then $K$ will be a common fundamental domain for $A(\Bbb Z^d)$ and $B^{-t}(\Bbb Z^d)$ and this will imply that the Fuglede-Gabor problem  holds. In particular, this is true when $A^tB$  is an integer matrix and $K$ is a bounded set, as our next result confirms.

\medskip

\begin{theorem}\label{lower triangle}  Let $K$ be  a bounded measurable subset of $ \Bbb R^d$  with  positive measure, and let  $\Lambda\subset \Bbb R^{2d}$ be a lower triangular  lattice in (\ref{lower_block}). Suppose that
 ${\mathcal G}(|K|^{-1/2}\chi_K, \Lambda)$ is a Gabor orthonormal basis for $L^2(\Bbb R^d)$ and  $A^tB$  is an integer matrix.  Then $K$ tiles and is spectral. More precisely, $K$ is a common fundamental domain for  $A(\Bbb Z^d)$ and $B^{-t}(\Bbb Z^d)$, $K$ tiles by $A(\Bbb Z^d)$ and is spectral with spectrum $B(\Bbb Z^d)$.
\end{theorem}

%Observe that the results of the theorem     holds particularly  if the block matrix  is a lower triangular symplectic matrix with $\det(M)=1$.
\medskip

% To prove  the theorem,    we show that   it is sufficient to show that if
%${\mathcal G}(|K|^{-1/2}\chi_K,  \Lambda)$ is a Gabor orthonormal basis for $L^2(\Bbb R^d)$ with  block matrices $A=B=I_d$,     the $d\times d$ identity matrix,
%then $K$ must tile $\Bbb R^d$  by ${\mathbb Z}^d$.  By the Fuglede Conjecture\rq{}s result for lattices (see the next section), $K$ is also    a spectral set with spectrum $({\mathbb Z^d})^\perp={\mathbb Z}^d$.
%Notice that we do not  assume any condition  on the matrix $C$.   In particular,  $C$ can be,  for example,  a matrix with irrational entries. Moreover, later in Theorem \ref{union of FD}  we show  the set $K$ is a union of fundamental domains of two different lattices independent of choice of the matrices $A$, $B$ or $C$. Yet $K$ might not be spectral  or tiling set (see \cite{KM06,Matolcsi,T04}).
%The tiling and spectral conclusion of the theorem only  depends on   $A^tB$. In other words, in this case  the lattice $\Lambda$ is   ``some sense" acting  like a separable lattice.

As a consequence of  Theorem \ref{lower triangle},  we resolve the  Fuglede-Gabor Problem \ref{our conjecture1}  in dimension one  for rational matrices  and $K$ is bounded.

\vskip.124in

\begin{theorem}\label{rational_dim1} Suppose that $K\subset \Bbb R$ is a bounded set with positive Lebesgue measure. Suppose that   $\Lambda$   is a rational lattice in $\Bbb R^{2}$ with $dens(\Lambda)=1$.  If
 ${\mathcal G}(|K|^{-1/2}\chi_K, \Lambda)$ is a Gabor orthonormal basis for $L^2(\Bbb R)$,  then $K$ tiles and is spectral.
\end{theorem}

\medskip

We  also have the following result for upper triangular block matrices using Theorem \ref{lower triangle}.

%As we  will observe later,  the Fuglede-Gabor Problem  is a subtle question for rational lattices in dimensions $d\geq 2$.   However, we have the following general result for upper triangular rational lattices.
                                    \vskip.124in

\begin{theorem}\label{UT} Suppose that $K\subset \Bbb R^d$ is a bounded set with positive Lebesgue measure. Supposes that  $\Lambda\subset \Bbb R^{2d}$ is a lattice such that
 $\Lambda = \left(
                                      \begin{array}{cc}
                                        A & D \\
                                        O & B \\
                                      \end{array}
                                    \right)({\mathbb Z}^{2d})$ with  $dens(\Lambda)=1$,
                                     $A^{-1}D$  symmetric rational matrix and $A^tB=I$. If
 ${\mathcal G}(|K|^{-1/2}\chi_K, \Lambda)$ is a Gabor orthonormal basis for $L^2(\Bbb R^d)$, then $K$ tiles and is spectral.
 \end{theorem}

 \medskip

Theorem \ref{lower triangle} may also be consider as a converse of \cite[Lemma 4.1]{HanW4}, which states that  if   $K$ is a  common fundamental domain for the lattice $A({\mathbb Z}^d)$ and $B^{-t}({\mathbb Z}^d)$, then for any  matrix $C$, the system  ${\mathcal G}(|K|^{-1/2}\chi_K,\Lambda)$ is an Gabor orthonormal basis. Therefore one may  naturally expect that $N=1$  in Theorem \ref{Th_union of FD} is always the case. However, we will show that  {\it $N>1$ can actually happen with a suitable choice of $C$ if $A^tB$ is a rational matrix (see Example \ref{mutli-tile K})}. This poses additional difficulty to solve the Fuglede-Gabor Problem for rational matrices in higher dimension, as we shall discuss it  later.  Finally, for a general matrix containing irrational entries, the answer to  Fuglede-Gabor problem is completely open.  We will discuss  this in detail  in Section \ref{Open problems}.

%In Theorem \ref{UT} we prove the  existence of a $d\times d$  integer matrix $E$ such that $K$  tiles  by lattice  $AE^{-t}(\Bbb Z^d)$  when $A^{-1}D$ a rational matrix and tiles by $A(\Bbb Z^d)$ when $A^{-1}D$ is an integer matrix.  % An easy calculation shows that under the assumption that  $A^{-1}D$  is an integer matrix and $K$ tiles by $A(\Bbb Z^d)$, then the converse of the theorem also holds.
 %
%The proof we provide  for  Theorem \ref{UT} is constructive.
%For this,
% once again, by the hypothesis of the theorem and appealing to Lemma \ref{lemma 3-prim} we can assume that $A=B=I_d$ and $D$ is rational and symmetric. Then by    reducing  the   upper triangle case into a lower triangle case through a special transformation,   the conclusion of the theorem follows by  Theorem \ref{lower triangle}.  In this concern,  Lemma \ref{complete residue classes} has a key role in the proof.
%

 \medskip

 {\bf Outline of the paper.} We organize the paper as follows: After some definitions and recalling some known and basic facts about  lattices and Gabor analysis in Section \ref{notations},  in Section \ref{proof of Theorem 1.3}  we prove Theorem \ref{Th_union of FD}.
 The proof of Theorems \ref{lower triangle} are \ref{rational_dim1} are presented in
   Section \ref{thm:lower triangle}.  In  Section \ref{thm:rational_dim1 and UT} we prove Theorem \ref{UT}.   Section \ref{Examples} is devoted to   examples illustrating the possibility for $N> 1$.
   We conclude the paper with a series of open problems in Section  \ref{Open problems} both  for rational and irrational lattices as well as the full generality of the Fuglede-Gabor Problem.  In our exposition, we discover that a new notion of completeness, which we will call  {\it exponential completeness},  is crucial in studying the Fuglede-Gabor problem, we will give a short study in Appendix A. In Appendix B, we will show that the octagon will not produce any Gabor orthonormal basis using rational matrices.

 \section{Preliminaries}\label{notations}

In this section, we will collect several basic definitions and results required for the rest of the paper.  A {\it full-rank lattice}
 $\Lambda\subset \Bbb R^d$ is a discrete and countable subgroup of $\Bbb R^{d}$ with compact quotient group  $\Bbb R^{d}/\Lambda$. A full-rank lattice  in $\Bbb R^d$ is  given by $\Lambda = {M}({\mathbb Z}^{2d})$ for some
  $2d\times 2d$ invertible matrix $M\in GL(2d, \Bbb R)$. The density of $\Lambda$ is given by   dens$(\Lambda)=|\det(M)|^{-1}$.

 Let $\Lambda$ be a lattice in $\Bbb R^d$.
   The {\it dual lattice} of $\Lambda$ is  defined as
 $$\Lambda^\perp:= \{ x\in \Bbb R^{2d} : \ \langle \lambda, x\rangle\in \Bbb Z,  \ \forall \lambda\in \Lambda\} .$$
A direct calculation shows that  $\Lambda^\perp=  M^{-t}(\Bbb Z^{d})$.

A fundamental domain of a lattice $\Lambda$ is a measurable set $\Omega$ in $\Bbb R^d$ which contains distinct representatives (mod $\Lambda$) in $\Bbb R^d$, so that the any intersection of $\
\Omega$ with any coset $x+\Lambda$ has only one element. For the existence of a fundamental domain see   \cite[Theorem 1]{Feld-Green68}.  It is also evident that $\Omega$ tiles $\Bbb R^d$ with translations by $\Lambda$  and
any other tiling set differs from  $\Omega$ at most for a zero measure set.

\noindent{\bf 1. A reduction lemma.} For an invertible $d\times d$  matrix
 $A$, the operator    ${\mathcal D}_A:L^2(\Bbb R^d)\to L^2(\Bbb R^d)$ defined by
 $${\mathcal D}_Ag(x):=|\det(A)| ^{1/2} g(Ax).$$
 is unitary, i.e, ${\mathcal D}_A$ is onto and isometry $\|{\mathcal D}_Ag\|= \|g\|$. The following lemma follows from a simple computation and is in general known. We will omit the detail of the proof. 

\vskip.124in

 \begin{lemma}\label{lemma 3-prim}
 Let ${\Lambda} $ be a lattice such that
\begin{align}\label{block lattice}
\Lambda= \left(
                    \begin{array}{cc}
                       A & D\\
                       C & B  \\
                    \end{array}
                  \right)({\mathbb Z}^{2d})
                   \end{align}
                   where  $A$ is an invertible $d\times d$  matrix. Then
$$\mathcal G(g, \Lambda) = {\mathcal D}_{A^{-1}}\mathcal G({\mathcal D}_Ag, \widetilde{\Lambda})$$
 where  $$
\widetilde{\Lambda} = \left(
                    \begin{array}{cc}
                       I & A^{-1}D\\
                       A^tC & A^tB  \\
                    \end{array}
                  \right) ({\mathbb Z}^{2d}).
 $$

Consequently,  $\mathcal G(g, \Lambda) $ is a Gabor orthonormal basis if and only if $\mathcal G({\mathcal D}_Ag, \widetilde{\Lambda})$ is a Gabor orthonormal basis.
   \end{lemma}

%   \begin{proof} The proof follows from a direct calculation. Indeed,
%   for any  $(m,n)\in \Bbb Z^d\times \Bbb Z^d$ and  $\lambda= \left(
%                    \begin{array}{cc}
%                       A & D\\
%                       C & B  \\
%                    \end{array}
%                  \right) (m,n)^t \in \Lambda$, we have
% %
%   \begin{align}\notag
%   \pi(\lambda)g(x)&=
%   \pi(Am+Dn, Cm+Bn) g(x)\\\notag
%   &= g(x-Am-Dn) e^{2\pi i \langle Cm+Bn,x\rangle}  \\\notag
%   &= |\det A|^{-1/2} {\mathcal D}_Ag(A^{-1} x-m-A^{-1}Dn) e^{2\pi i \langle A^tCm+A^tBn,A^{-1}x\rangle} \\\notag
%   &= {\mathcal D}_{A^{-1}} \pi(m+A^{-1}Dn, A^tCm+A^tBn){\mathcal D}_Ag(x) .
%   \end{align}
%   The conclusion of the lemma now follows immediately from the preceding equalities  and the fact that  orthonormal bases are preserved under  the unitary maps.
%   \end{proof}

In Lemma \ref{lemma 3-prim},  if we let
   $D=O$ and $g(x)=|K|^{-1/2} \chi_K$, then the conclusion of the lemma shows that   ${\mathcal G}(|K|^{-1/2} \chi_K, \Lambda)$ is an orthonormal basis  for $L^2(\Bbb R^d)$ if and only if  ${\mathcal G}(|A^{-1}K|^{-1/2} \chi_{A^{-1}K}, \widetilde{\Lambda})$ is an orthonormal basis with
$$
\widetilde{\Lambda} = \left(
                    \begin{array}{cc}
                       I & O\\
                       A^tC & A^tB  \\
                    \end{array}
                  \right) ({\mathbb Z}^{2d}).
$$
 We shall use this observation   later.

 \iffalse
\begin{proof}
Let $g = |K|^{-1/2} \chi_K$. For each $\lambda = (Am,Cm+Bn),\lambda' =(Am',Cm'+Bn')\in \Lambda$. Note that by a change of variable $x \rightarrow Ax$
$$
\begin{aligned}
\langle\pi(\lambda)g,\pi(\lambda')g\rangle =&  |K|^{-1} \int \chi_K (x-Am)\chi_K (x-Am') e^{-2\pi i \langle C(m-m')+B(n-n'),x\rangle}dx\\
=& |K|^{-1}\int \chi_K (x-Am)\chi_K (x-Am') e^{-2\pi i \langle C(m-m')+B(n-n'),x\rangle}dx\\
=&  |\det A||K|^{-1}\int \chi_K (Ax-Am)\chi_K (Ax-Am') e^{-2\pi i \langle C(m-m')+B(n-n'),Ax\rangle}dx\\
=& |A^{-1}K|^{-1} \int \chi_{A^{-1}K} (x-m)\chi_{A^{-1}K} (x-m')e^{-2\pi i \langle A^tC(m-m')+A^tB(n-n'),x\rangle}dx.
\end{aligned}
$$
This shows that ${\mathcal G}(|K|^{-1/2} \chi_K, \Lambda)$ is mutually orthogonal if and only if ${\mathcal G}(|A^{-1}K|^{-1/2} \chi_{A^{-1}K}, \widetilde{\Lambda})$ is mutually orthogonal. In a similar change of variable, we see that the Parseval identities of these Gabor system are actually equivalent. Hence, the statement of the lemma follows.
\end{proof}
\fi

\medskip

\noindent{\bf 2. Orthogonality implies completeness.} The following proposition says that completeness automatically holds for a lattice of density one if we can establish the mutually orthogonality.

\medskip

                   \begin{proposition}\label{complt}
                   Let $g\in L^2(\Bbb R^d)$, $\|g\|=1$ and $\Lambda\subset \Bbb R^{2d}$ be a lattice with density $\mbox{dens}(\Lambda)=1$. Assume that
                   $\mathcal G(g,\Lambda)$ is  an orthonormal set. Then $\mathcal G(g,\Lambda)$ is complete.
                  \end{proposition}

\medskip
                                     For the  proof of Proposition \ref{complt} we require  the following lemma.
                                  Note that for a positive Borel measure $\mu$,
                                  $$
                                  f\ast\mu (x) = \int f(x-y)d\mu(y),
                                  $$
                                  given that the integral is well-defined. If
 $\mu = \sum_{\lambda\in\Lambda}\delta_{\lambda}$, then $\chi_K\ast\mu = 1$ ($\le 1$) if and only if $K$ tiles (packs) ${\mathbb R}^d$ by $\Lambda$\footnote{Recall that $K$ packs ${\mathbb R}^d$ by ${\mathcal J}$ if $\sum_{t\in{\mathcal J}}\chi_{K}(x-t)\le1$, a.e. $x\in \Bbb R^d$}. With this introduction we recall the following  result.

\medskip

                   \begin{lemma}\label{GLW}(\cite[Theorem 2.1]{GLW})
                 Suppose that $f,g\in L^1({\mathbb R}^d)$ are non-negative functions such that $\int f(x) dx = \int g(x) dx= 1$. Suppose that for positive Borel measure $\mu$  on $\Bbb R^d$
$$
f\ast\mu \le 1  \ \mbox{and} \ g\ast\mu\le 1.
$$
Then $f\ast\mu = 1$  if and only if $ g\ast \mu= 1$.
\end{lemma}

 Given $f,g\in L^2({\mathbb R}^d)$, the {\it short time Fourier transform} is defined by
        \begin{align}\label{STFT}
V_gf(t,\xi)=\int f(x)\overline{g(x-t)}e^{-2\pi i \langle\xi,x\rangle}dx, \quad (t,\xi)\in \Bbb R^{2d}
 \end{align}
and it is  a continuous function on ${\mathbb R}^{2d}$ \cite{Groechenig-book}.

                  \begin{proof}[Proof of Proposition \ref{complt}]
                  The mutual orthogonality of $\mathcal G(g, \Lambda)$ implies the Bessel inequality of the system:
\begin{align}\label{Bessel}
\sum_{(t,\xi)\in \Lambda} |V_gf(t, \xi)|^2\leq \|f\|^2 \ , \quad \forall f\in L^2(\Bbb R^d).
\end{align}
Let $ s,\xi  \in \Bbb R^d$. The inequality (\ref{Bessel}) for $e^{2\pi i \langle \nu, x\rangle} f(x-s)$ in the place of $f$ yields
\begin{align}\notag
\sum_{(t,\xi)\in \Lambda} |V_gf(t-s, \xi-\nu)|^2\leq \|f\|^2  \quad \forall \ (s, \nu)\in \Bbb R^{2d}.
\end{align}
Hence, $|V_gf|^2\ast \delta_\Lambda \leq \| f\|^2$. Take $G:= \| f\|^{-2}|V_gf|^2$. Then $\int_{\Bbb R^{2d}}  G(z) dz=1$ and $G\ast \delta_\Lambda\leq 1$. On the other hand, $\Lambda$ is a lattice with  density $1$. Let $\Omega\subset \Bbb R^{2d}$ be any fundamental domain for $\Lambda$. Then  $|\Omega|=1$ and it tiles $\Bbb R^{2d}$ by  $\Lambda$. Therefore $\chi_\Omega\ast \delta_\Lambda=1$. Now Lemma \ref{GLW}  implies that    $G\ast \delta_\Lambda =1$. But this  is equivalent to the completeness of the system $\mathcal G(g, \Lambda)$ and we are done.
                  \end{proof}

\medskip

\noindent{\bf 3. Some reduction to lower triangular block matrices.} The following result  is due to Han and Wang which states that any  invertible integer matrix   can be converted into a  lower triangular integer  matrix. We will need it in later sections. An integer matrix $P$ is called {\it unimodular}  if $\det P=1$.

\medskip

\begin{lemma}\label{HanW1}(\cite{HanW4}, Lemma 4.4) Let $M$ be an $d\times d$ invertible integer matrix. Then there is an  $d\times d$  unimodular  matrix $P$ such that $MP$ is a  lower triangular integer matrix.
\end{lemma}

%\begin{proof} Assume that $d_1$ is the least common divisor of the first row of $M$, i.e.  $d_1=gcd(m_{1j}, 1\leq j\leq d)$. By a simple fact from number theory, there are integers  $x_i$ such that $gcd(x_j)=1$ and $x_1 m_{11}+x_2m_{12}+\cdots+x_dm_{1d}=d_1$. Let  $P_1$ be the unimodular matrix  whose  first column is $(x_1, \cdots, x_d)^t$. (For the existence of such matrix consult \cite[Theorem II 1]{Ne}.)  Put $M_1=MP_1$. It is obvious that the $(1,1)$-entry of the matrix $M_1$ is $d_1$ and other entries of the matrix on the first row are multiples of $d_1$. Thus by the Gaussian elimination method, the matrix $M_1$ can be reduced to a matrix where the entries of $M_1$ on the first row are all replaced by $0$ and the $(1,1)$-entry is $d_1$. Let us call this matrix $M_2$.  Note that there is a unimodular matrix $P_2$ such that $M_2= M_1P_2$.   By the  induction one can show that there is finite many unimodular integer matrices $P_i$ such that $MP_1P_2\cdots P_{k}$ is a lower triangular integer matrix, as desired.
%\end{proof}

As a corollary of Lemma  \ref{HanW1} we can show that any rational matrix can be represented as a  lower triangular rational matrix.

\medskip

\begin{corollary}\label{M=N}
Let $M$ be an $d\times d$ invertible rational matrix. Then there is a    lower triangular rational matrix $N$ such that $M(\Bbb Z^d) = N(\Bbb Z^d)$.
\end{corollary}

Henceforth, we shall say matrix $M$ is {\it equivalent} to $N$ if $M(\Bbb Z^d) = N(\Bbb Z^d)$.

\medskip

\noindent{\bf 4. Exponential Completeness.}\label{exponential completeness} Recall that a collection of functions $\{\varphi_n\}$ is said to be {\it complete} in $L^2(\Omega)$ if $\langle f, \varphi_n\rangle =0$ for all $n$ implies that $f=0$ a.e. on $L^2(\Omega)$. Given $f\in L^2({\mathbb R}^d)$, the Fourier transform of $f$ is defined to be $\widehat{f}(\xi)= \int_{{\mathbb R}^d}f(x)e^{-2\pi i \langle\xi, x\rangle}dx$. In our study, we will need to following weaker notion of the completeness property. 

\medskip
 
\begin{definition}\label{exponential completeness} 
 Let $\Lambda$ be a countable set and let $\Omega$ be a Lebesgue measurable set with positive finite measure. We say that the set of exponentials $\{e^{2\pi i \langle\lambda,x\rangle}:\lambda\in\Lambda\}$ (or  $\Lambda$) is {\it exponentially complete} for $L^2(\Omega)$ if there does not exist any $\xi\in{\mathbb R}^d$ such that 
 %an exponential function $e^{2\pi i \langle \xi,x\rangle}$ such that 
 $$
 \widehat{\chi_{\Omega}}(\lambda-\xi)=\int_{\Omega}  e^{2\pi i \langle \xi ,x\rangle}  e^{-2\pi i \langle \lambda, x\rangle}  dx =0, \ \forall \lambda\in\Lambda.
 $$
 \end{definition}

\medskip

\begin{remark}\label{remark2}
Throughout the paper, we will see that exponential completeness plays an important role in constructing Gabor orthonormal basis using non-separable lattices.  If a countable set of exponentials is complete for $L^2(\Omega)$, then it must be exponentially complete (otherwise $e^{2\pi i \langle \xi,x\rangle}$ will be orthogonal to all $e^{2\pi i \langle\lambda,x\rangle}$ contradicting completeness). However, the converse is not true.
 For example, the set of exponentials associated to the lattice $\Lambda=\sqrt{2}{\mathbb Z}$ is exponentially complete in $L^2([0,1])$, but it is not complete in it (see Lemma \ref{lemma 4}).   In Appendix 1, we will give a short study about the exponential completeness for lattices in $L^2[0,1]^d$.
\end{remark}

\section{Proof of Lemma \ref{Th_union of FD} - Union of fundamental domains}\label{proof of Theorem 1.3} 
 We now prove our Theorem \ref{Th_union of FD}. It follows from two theorems in two separate fields. The first one is taken from the study of Fuglede's problems.  It was first proved by Jorgensen and Pedersen \cite[Theorem 6.2 (b)]{JoPe}  and then Lagarias and Wang \cite[Theorem 2.1]{LW97} gave a simpler proof.

\medskip

\begin{theorem}\label{thJoPeLgWa}
Let $\Omega\subset{\mathbb R}^d$ be a  Lebesgue measurable set with positive finite measure. Suppose $\Gamma$ is a full-rank lattice such  that $\Gamma \subseteq \{\xi: \widehat{\chi_{\Omega}}(\xi) = 0\} \cup \{0\}$. Then
$
\Omega = \bigcup_{j=1}^N D_j$, up to measure zero, where $D_j$ are  fundamental domains for the dual lattice $\Gamma^\perp$, $|D_i\cap D_j|=0, \ i\neq j$ ~ and $N=|\Omega|/|D_j|$.
 \end{theorem}
%Let $\Gamma=B(\Bbb Z^d)$ be a full lattice in $\Bbb R^d$ for some $d\times d$ invertible matrix $B$.
%As a consequence of the preceding theorem, notice that if $|\Omega||\det(B)|=1$ and   $B({\mathbb Z}^d)\setminus\{0\}\subseteq \{\xi: \widehat{\chi_{\Omega}}(\xi) = 0\}$, then  $\Omega$ is a fundamental domain of  $(B^t)^{-1}({\mathbb Z}^d)$ which also tiles.

\medskip

Given a lattice  $\Lambda=M(\Bbb Z^{2d})$, the   {\it adjoint lattice} $\Lambda^\circ$ is   a lattice such that
                  $$J(\Lambda^\circ)=  \Lambda^\perp $$
                  where
             $J=  \left(
                    \begin{array}{cc}
                      O & -I \\
                       I & O \\
                    \end{array}
                  \right)$. In other words, $\Lambda^\circ= J^{-1}M^{-t}(\Bbb Z^d)$.

The  Ron-Shen duality theorem \cite{RS-duality97} is well-known in Gabor analysis. It was first proved over symplectic lattices, it is known to be true over any lattice (see e.g \cite[Theorem 2.3]{Groechenig-mystery} for a proof by Poisson Summation Formula). We will need the following version of duality theorem.

\medskip

\begin{theorem}\label{thRS} ${\mathcal G}(g,\Lambda)$ is a Gabor orthonormal basis if and only if ${\mathcal G}(g,\Lambda^{\circ})$ is a Gabor orthonormal basis.
\end{theorem}

\begin{proof}[Sketch of Proof.] This statement is  well-known. Here we provide a simple proof based on \cite[Theorem 2.3]{Groechenig-mystery} and Proposition \ref{complt}. Since $(\Lambda^{\circ})^{\circ} =\Lambda$, both sides of the statements are symmetric and we just need to prove one side of the equivalence. Suppose that ${\mathcal G}(g,\Lambda)$ is a Gabor orthonormal basis. Then $g$ is the only dual window with the property that 
$$
\langle g, \pi (\mu)(g)\rangle = 0 \ \forall \mu\in\Lambda^{\circ}\setminus \{0\}
$$
(by \cite[Theorem 2.3]{Groechenig-mystery}). This means that for all distinct $\mu,\mu'\in\Lambda^{\circ}$,  $\langle \pi(\mu)g, \pi (\mu')(g)\rangle = c \langle g, \pi (\mu-\mu')(g)\rangle = 0$ ($c$ is some unimodular constant). Thus ${\mathcal G}(g,\Lambda^{\circ})$ is mutually orthogonal. As ${\mathcal G}(g,\Lambda)$ is a Gabor orthonormal basis, $\|g\|=1$ and dens$(\Lambda^{\circ})$ = dens$(\Lambda)$ =1, by Proposition \ref{complt}, ${\mathcal G}(g,\Lambda^{\circ})$ is complete and is thus an orthonormal basis.
\end{proof}

For a lower triangular lattice $\Lambda=  \left(
                                      \begin{array}{cc}
                                        A & O \\
                                        C & B \\
                                      \end{array}
                                    \right){\mathbb Z}^{2d}$,  the adjoint lattice  $\Lambda^\circ$  is also a lower triangular and it can be calculated as follows:
$$
\Lambda^{\circ} = \left(
                    \begin{array}{cc}
                      O & I \\
                      -I & O \\
                    \end{array}
                  \right)\left(
                    \begin{array}{cc}
                      A^{-t} & -A^{-t}C^tB^{-t} \\
                      O & B^{-t} \\
                    \end{array}
                  \right)  ({\mathbb Z}^{2d}) = \left(
                    \begin{array}{cc}
                      O & B^{-t} \\
                      -A^{-t} & A^{-t}C^tB^{-t} \\
                    \end{array}
                  \right) ({\mathbb Z}^{2d}).
$$

From the other hand we can write  $$\left(
                    \begin{array}{cc}
                      O & B^{-t}\\
                      -A^{-t}& A^{-t}C^tB^{-t} \\
                    \end{array}
                  \right)  =  \left(
                    \begin{array}{cc}
                       B^{-t} & O\\
                       A^{-t}C^tB^{-t} & A^{-t}  \\
                    \end{array}
                  \right)
                  \left(
                    \begin{array}{cc}
                      O & I \\
                      -I & O\\
                    \end{array}
                  \right)({\mathbb Z}^{2d}).
$$
But  $\left(
                    \begin{array}{cc}
                      O & I \\
                      -I & O\\
                    \end{array}
                  \right) ({\mathbb Z}^{2d})={\mathbb Z}^{2d}$, therefore we have
$$
\Lambda^{\circ}= \left(
                    \begin{array}{cc}
                       B^{-t} & O\\
                       A^{-t}C^tB^{-t} & A^{-t}  \\
                    \end{array}
                  \right) ({\mathbb Z}^{2d}).
$$

\begin{proof}[Proof of Theorem \ref{Th_union of FD}]
The orthogonality of the Gabor system  implies that
 \begin{align}\label{ortho}
 \int_K e^{-2\pi i \langle  Bn, x\rangle} dx = 0 \quad \forall \ n\in \Bbb Z^d\setminus\{0\} .
 \end{align}
 By Theorem \ref{thJoPeLgWa}, (\ref{ortho}) implies      that  $K$ can be written as
 $
 K = \bigcup_{j=1}^N D_j,
 $
 where $D_j$  is a fundamental domain  for $B^{-t}({\mathbb Z}^d)$. On the other hand,  by the duality Theorem \ref{thRS},  ${\mathcal G}({|K|^{-1/2}}\chi_K,\Lambda^{\circ})$ is a Gabor orthonormal basis, too.  Similarly,   the exponentials $\{e^{2\pi i  \langle A^{-t}n, x\rangle}: n\in {\mathbb Z}^d\}$ are mutually orthogonal  in $L^2(K)$. Hence, by Theorem \ref{thJoPeLgWa}, we have
 $
 K = \bigcup_{j=1}^M E_j
 $
 where $E_j$\rq{}s are fundamental domains for $A({\mathbb Z}^d)$.  Since $\det(AB)=1$  we conclude that  $|D_i|=|E_j|, \forall \ i, j$. Since $|K|<\infty$, the latter forces that $M=N$,  hence the proof of the theorem is completed.
\end{proof}

\section{Proof of Theorems \ref{lower triangle} and \ref{rational_dim1} - Lower triangular matrices}\label{thm:lower triangle}

%We will consider the lower triangular matrices in this section and prove Theorem \ref{lower triangle}.
To prove Theorem \ref{lower triangle},  first we shall apply  some preliminary reductions to  the theorem, as follows.   Due to   Lemma  \ref{lemma 3-prim}  and  the hypothesis of  the theorem  on the matrices $A$ and $B$,   for the proof  it is sufficient
  to  assume that
$
\Lambda = \left(
                    \begin{array}{cc}
                       I & O\\
                       C & B\\
                    \end{array}
                  \right) ({\mathbb Z}^{2d}) ,
 $
 where $B$ is an invertible matrix with integer entries (since originally $A^tB$ has integer entries by the assumption of Theorem \ref{lower triangle}). Notice by the density condition dens$(\Lambda) = 1$,  we have $|\det(B)|=1$. Thus $B^{-1}$ is also an integral matrix with determinant 1  and we have
  $\Bbb Z^{2d} = \left(
                    \begin{array}{cc}
                       I & O\\
                       O & B^{-1}\\
                    \end{array}
                  \right) ({\mathbb Z}^{2d})$. Thus
      we can rewrite $\Lambda$ as follows:
 $$
\Lambda = \left(
                    \begin{array}{cc}
                       I & O\\
                       C & B\\
                    \end{array}
                  \right) \left(
                    \begin{array}{cc}
                       I & O\\
                       O & B^{-1}\\
                    \end{array}
                  \right) ({\mathbb Z}^{2d}) = \left(
                    \begin{array}{cc}
                       I & O\\
                       C & I\\
                    \end{array}
                  \right) ({\mathbb Z}^{2d}).
                  $$
Therefore to prove Theorem \ref{lower triangle}   it suffices to consider $\Lambda = \left(
                    \begin{array}{cc}
                       I & O\\
                       C & I\\
                    \end{array}
                  \right) ({\mathbb Z}^{2d})$. In this case  if ${\mathcal G}\left(|K|^{-1/2}\chi_K,\Lambda\right)$ is an orthonormal basis,  then  by Theorem \ref{Th_union of FD} we must have
 $
K = \bigcup_{j=1}^N E_j
 $
where $E_j$\rq{}s are disjoint fundamental domains of ${\mathbb Z}^d$. In particular, $K$ is a multi-tile for $\Bbb R^d$ with respect to ${\mathbb Z}^d$, i.e.,
\begin{equation}\label{multitile}
\sum_{k\in{\Bbb Z}^d} \chi_K(x+k) = N \ \mbox{a.e.} \ x\in [0,1)^d.
\end{equation}
 Our goal is to show that $N=1$. For this, the following proposition will serve a key role.
  \medskip

\begin{proposition}\label{prop2} Suppose $K$ is a bounded set which multi-tiles $\Bbb R^d$ with respect to $\Bbb Z^d$ at level $N$,  i.e. (\ref{multitile}) holds.  Suppose that $N>1$. Then there exists $m \in{\mathbb Z}^d$ such that
\begin{enumerate}
\item $K\cap (K+m)$ has positive Lebesgue measure.
\item $ K\cap (K+m)$ consists of distinct representative   (mod ${\mathbb Z}^d$).
\item $K\cap (K+m)$ is a packing by ${\mathbb Z}^d$, i.e.
$$
\sum_{n\in {\mathbb Z}^d} \chi_{K\cap (K+m)} (x+n)\le 1. \ a.e.
$$
\end{enumerate}
\end{proposition}
 \begin{proof}
Put  $Q := [0,1)^d$. The identity (\ref{multitile}) means that for a.e. $x\in Q$, there are exactly $N$ integers $n_1,...,n_N$ such that  $x+n_i\in K$ $i=1,...,N$.  Using this observation we will  decompose $K$ as follows. For each $x\in Q$, put
$$
K(x) := \{k\in \Bbb Z^d: x+ k\in K\}.
$$
Then $|K(x)|=N$.
 Let $S \subset{\mathbb Z}^d$ and $|S| = N$. Define
$$
K_S = \{x\in Q: K(x) = S\}.
$$
Since $K$ is a multi-tile of level $N$, we have
$$
K = \bigcup_{|S|=N} (K_S+S), \ \mbox{and} \ \bigcup_{|S|=N} K_S = Q = [0,1)^d,
$$
where the union runs through all possible subsets $S\in{\mathbb Z}^d$ of cardinality $N$. Furthermore,
 $K_S \cap K_{S'} =\emptyset, \forall S\ne S'
$, since there are exactly $N$ integers $n$ for which $x+n\in K$. By the boundedness of $K$, there are only finitely many possible $S\in{\mathbb Z}^d$ with $|S|=N$ such that $|K_S|>0$. Thus, we can enumerate those $S$ as $S_1,...,S_r$ so that
\begin{equation}\label{decomp}
K = \bigcup_{i=1}^r (K_{S_i}+S_i).
\end{equation}

(Notice the   decomposition (\ref{decomp}) also holds for any multi-tile bounded set $K$  with respect to any lattice $\Gamma$ in place of $\Bbb Z^d$ and
 any bounded fundamental set of  $\Gamma$ in the place of $Q$.)

We order ${\mathbb Z}^d$ by the natural lexicographical ordering. We then enumerate all possible elements in $S_i$, $1\le i\leq r$, by
$$
 n_1^{S_i}<....<n^{S_i}_{N}
$$
Let
$$
 m_{S_i} =  n_N^{S_i}- n_1^{S_i}, \ \mbox{and} \  m= \max_{i=1,...,r} m_{S_i}.
$$
Note that for any $i\ne j$, since $K_{S_i}$ and $K_{S_j}$ are distinct subsets in $Q$, the intersection of  $K_{S_i}+{\mathbb Z}^d$ and $K_{S_j}+{\mathbb Z}^d$ has zero Lebesgue measure. As  $K_{S_i}+S_i\subset K_{S_i}+{\mathbb Z}^d$ and $K_{S_j}+S_j+ m\subset K_{S_j}+{\mathbb Z}^d$, the set $(K_{S_i}+S_i)\cap (K_{S_j}+S_j+ m)$ has zero Lebesgue measure for $i\neq j$. Therefore, up to Lebesgue measure zero, we have
$$
\begin{aligned}
K\cap(K+ m) =& \left(\bigcup_{i=1}^r (K_{S_i}+S_i)\right)\cap\left(\bigcup_{j=1}^r (K_{S_j}+S_j+m)\right)\\
=&\bigcup_{i=1}^r (K_{S_i}+S_i)\cap (K_{S_i}+S_i+ m)\\
\end{aligned}
$$
Note that $(K_{S_i}+S_i)\cap (K_{S_i}+S_i+ m)$ has positive Lebesgue measure if and only if $S_i\cap (S_i+{m})\ne \emptyset.$ By our construction of $ m$,
$$
S_i\cap (S_i+{m}) = \left\{
                  \begin{array}{ll}
                    \emptyset, & \hbox{if $m_{S_i}<m$;} \\
                    n_N^{S_i}, & \hbox{if $m_{S_i}= m$.}
                  \end{array}
                \right.
$$
Hence,
$$
K\cap(K+m) = \bigcup_{\{i: m_{S_i} = m\}} (K_{S_i}+  n_N^{S_i}).
$$
Note that $\{i :m_{S_i}= m\}$ is non-empty by our construction, thus $K\cap (K+m)$ has positive Lebesgue measure. On the other hand, $K_{S_i}$ consists of distinct representatives in $Q$. Therefore $K\cap(K+m)$ has distinct representatives (mod ${\mathbb Z}^d$) in $Q$. This implies (2). The conclusion (3) follows directly  from (2)  and  by the  definition of packing.
\end{proof}

\medskip
We remark that Proposition \ref{prop2} does not hold if $K$ is unbounded (see Example \ref{example3}).  The following well-known lemma also hold  (see \cite[Lemma 2.10]{GLi}).

 \medskip

\begin{lemma}\label{lemma 4}
Let $B$ be an $d\times d$ invertible matrix and $\Omega\subset \Bbb R^d$ be  a measurable set with positive and finite measure. The set $\{e^{2\pi i \langle B n, x\rangle}: n\in{\mathbb Z}^d\}$ is complete in $L^2(\Omega)$ if and only if $\Omega$ is a packing by $B^{-t}{\mathbb Z}^d$, i.e.
$$
\sum_{n\in {\mathbb Z}^d} \chi_{\Omega} (x+B^{-t}n)\le 1. \ \  a.e.\  x\in B^{-t}(Q) .
$$
Here, $Q=[0,1)^d$.
\end{lemma}

We are now ready to prove  Theorem \ref{lower triangle}.

 \begin{proof}[Proof of Theorem \ref{lower triangle}] As we mentioned above,
by   Lemma  \ref{lemma 3-prim} and the assumption that $A^tB$ is an integral matrix with $\det(A^tB)=1$,  it is sufficient to prove the theorem for lattices  in the form of
$$
\Lambda = \left(
                    \begin{array}{cc}
                       I & O\\
                       C & I  \\
                    \end{array}
                  \right) ({\mathbb Z}^{2d}).
$$
In this case, by the orthogonality of $\mathcal G(|K|^{-1/2}\chi_K, \Lambda)$,
the exponentials $\{e^{2\pi i \langle n, x\rangle}: \ n\in \Bbb Z^d\}$  are mutual  orthogonal in $L^2(K)$. Thus,
by Theorem \ref{Th_union of FD}, $K = \bigcup_{j=1}^N E_j$, where $E_j$\rq{}s are fundamental domains of ${\mathbb Z}^d$. 

\medskip

We claim that $N=1$. Suppose that $N>1$.   By  Proposition \ref{prop2}, there exists ${m}\in{\mathbb Z}^d$ such that  $K\cap (K+{m})$ has positive Lebesgue measure and $K\cap (K+{ m})$ is a packing  in $\Bbb R^d$ by ${\mathbb Z}^d$. Thus  by Lemma \ref{lemma 4}, the set $\{e^{2\pi i \langle  n,x\rangle}: n\in{\mathbb Z}^d\}$ is complete in $K\cap (K+{m})$. By Remark \ref{remark2}, ${\mathbb Z}^d$ is exponentially complete in $L^2(K\cap (K+{m}))$. Obviously, ${ m}=0$ cannot satisfy the packing property (3) in Proposition \ref{prop2} since we have assumed $N>1$. Thus  ${m}\ne 0$. On the other hand, the orthogonality of the Gabor system implies that for any $n\in \Bbb Z^d$ we must have
\begin{align}\label{eq}
   \widehat{\chi_{K\cap (K+{m})}} (C{m}+n)= \int_{\Bbb R^d} \chi_K(x)\chi_{K}(x-{m}) e^{-2\pi i \langle C { m}, x\rangle}e^{-2\pi i \langle  n, x\rangle}dx  =0
 \end{align}
 for all $n\in {\mathbb Z}^d$. 
This contradicts the exponential completeness of ${\mathbb Z}^d$. Therefore, the assumption $N>1$ cannot hold,  and $K$ is thus a fundamental domain of ${\mathbb Z}^d$. This completes the proof.
\end{proof}

\begin{remark} In proving Theorem \ref{lower triangle}, once $\Lambda$ is reduced to the lattice 
$${\mathcal A}({\mathbb Z}^{2d}), \ \mbox{where} \  {\mathcal A}=  \left(
                    \begin{array}{cc}
                       I & O\\
                       C & I  \\
                    \end{array}
                  \right),
$$
one may suspect that we can apply a metaplectic transformation to further reduce the lattice to the separable lattice ${\mathbb Z}^{d}\times {\mathbb Z}^d$ and hence the problem is trivially solved. Unfortunately, this approach does not seem to work. To be precise, one can look up \cite{Groechenig-book} for the metaplectic transformation in this case. We have that ${\mathcal G}(|K|^{-1/2}\chi_K,{\mathcal A}({\mathbb Z}^{2d}))$ is a Gabor orthonormal basis if and only if ${\mathcal G}( \tilde{g},{\mathbb Z}^{d}\times{\mathbb Z}^d)$ is a Gabor orthonormal basis, where 
$$
\tilde{g} = e^{-\pi  i\langle x,Cx\rangle}|K|^{-1/2}\chi_K.
$$
Although the lattice is separable, the window funciton $\tilde{g}$ is now complex-valued, we cannot conclude that $K$ has  no overlap in the time domain as it were the case when $C=O$. 

\smallskip

Nonetheless,  metaplectic transformation and symplectic matrices seems to provide some strong tools that may lead to a progess in the Fuglede-Gabor problem,  readers may refer to \cite{Folland,deGosson,Groechenig-book} for details about these tools.
\end{remark}

\medskip

 Theorem \ref{rational_dim1} is  now straightforward. We prove  it   here for the sake of completeness. 

\begin{proof}[Proof of Theorem \ref{rational_dim1}]
Let $\Lambda= M(\Bbb Z^2)$ where $M=\left(
                                      \begin{array}{cc}
                                       a & d \\
                                        c & b \\
                                      \end{array}
                                    \right)$  is a rational matrix with $\det M=1$. Let $q$ be the least common multiple of $a, b, c $ and $d$. Then we can write $\Lambda$ as  $\Lambda= q^{-1} \tilde M(\Bbb Z^2)$ where  $\tilde M$ is an integer matrix. By Lemma \ref{HanW1},  we can find a  unimodular integer matrix  $P$  such that $\tilde M P$ is the lower triangular integer matrix.  By the unimodularity of the matrix $P$  we have
                                     $q^{-1} \tilde M(\Bbb Z^2) = q^{-1}  \tilde M P(\Bbb Z^2)$. Therefore, $\Lambda = q^{-1}  \tilde M P(\Bbb Z^2)$ and  $q^{-1}  \tilde M P$ is a lower triangular rational matrix. We can therefore write
                                     $$
                                     \Lambda = \left(
                                      \begin{array}{cc}
                                       \alpha & 0 \\
                                        \gamma & \beta \\
                                      \end{array}
                                    \right)(\mathbb Z^2)$$
                                    for some $\alpha, \beta, \gamma\in \Bbb Q$.
                                    Notice that the density of $\Lambda$  equals $1$, meaning that $\alpha\beta = 1$. Thus, all assumptions of Theorem \ref{lower triangle} are satisfied. Hence, $K$ is a translational tile with tiling set $\Bbb Z$ and is  a spectral set. 
\end{proof}

 \section{  Proof of Theorem \ref{UT} - Upper triangular matrices}\label{thm:rational_dim1 and UT}
We will discuss a case of upper triangular matrices which can be converted into  the lower one. Then we will use   Theorem \ref{lower triangle} to prove the theorem.  First we need few lemmas.
\medskip

\begin{lemma}\label{Gamma-D}
    Let $D$  be  a $d\times d$  rational matrix.
    Then there is an integer  lattice $\Gamma$ such that $D\gamma\in \Bbb Z^d$ for $\gamma\in \Gamma$.
    \end{lemma}
    \begin{proof}
 Define
$$
\Gamma := \{k\in{\mathbb Z}^d: Dk\in{\mathbb Z}^d\}.
$$
It is easy to check that $\Gamma$ is a lattice contained in $\Bbb Z^d$.  Moreover, $\Gamma$  contains $p\Bbb Z^d$ where
 $p$  is the least common multiple of the denominators of   entries of $D$.  Thus, $\Gamma$ is  a full-rank lattice and has the form of
$
\Gamma = M({\mathbb Z}^d),
 $
for some  invertible $d\times d$ matrix  $M$ with integer entries.
\end{proof}
Observe that according to the  Lemma \ref{Gamma-D},  for any given rational matrix $D$, there is an integer $M$ such that $DM$ is integer.
   With this observation, we have the following result.

 \medskip

\begin{lemma}\label{complete residue classes} Let $D$ be a rational matrix, and  let   $\Gamma$  and $M$   be given  as in Lemma \ref{Gamma-D} and  $\det M=n$.
Suppose that $\{\gamma_1,\cdots, \gamma_n\}$ be a complete representative (mod $M({\mathbb Z}^d)$) in ${\mathbb Z}^d$.  If $D$ is symmetric, then   $\{M^tD\gamma_1,\cdots,$ $ M^tD\gamma_n\}$ is a complete representative (mod $M^t{\mathbb Z}^d$) in ${\mathbb Z}^d$.
\end{lemma}
\begin{proof} We saw above that by
  the structure of $M$ and $\Gamma$, $DM$ is an integer matrix, therefore $DM(\mathbb Z^d)\subseteq{\mathbb Z}^d$. We also have  $(DM)^t{\mathbb Z}^d\subseteq{\mathbb Z}^d$. This implies that
$$M^tD({\mathbb Z}^d) = (D^tM)^t({\mathbb Z}^d )\subseteq{\mathbb Z}^d.$$
Thus, $M^tD\gamma_i$ are all integer vectors for all $i = 1,\cdots,n$.

\medskip

Next, we show that $\{M^tD\gamma_1,...,M^tD\gamma_n\}$ consists of distinct representative (mod $M^t({\mathbb Z}^d)$) in $\Bbb Z^d$. Suppose that this is not the case. Then for some $i\neq j$ we must have  $M^tD\gamma_i - M^tD\gamma_j \in M^t({\mathbb Z}^d)$. This implies that $D\gamma_i-D\gamma_j\in {\mathbb Z}^d$, which means that  $\gamma_i$ and $\gamma_j$ belong  to the same representative class (mod $M({\mathbb Z}^d)$) in ${\mathbb Z}^d$ which is a contradiction.

\medskip

Finally, we show that $\{M^tD\gamma_1,...M^tD\gamma_n\}$ is complete.  This follows immediately by counting the  number of
cosets   present in ${\mathbb Z}^d/M^t({\mathbb Z}^d)$ which is  $|\det(M^t)| = |\det(M)|= n$. Hence, $\{M^tD\gamma_1,...M^tD\gamma_n\}$ is complete.
\end{proof}

In the following constructive lemma we  shall present   a class of   upper triangle lattice  which can be converted  into a lower triangular lattice.

  \begin{lemma}\label{a technical lemma}
            For   $\Lambda=
 \left(
                                      \begin{array}{cc}
                                        I & D \\
                                        O & I \\
                                      \end{array}
                                    \right)(\Bbb Z^{2d})$  with $D$  rational and symmetric, there are integer matrices $E$ and $X$ such that
         $\Lambda = \left(
                                      \begin{array}{cc}
                                       E^{-t} & O \\
                                        X & E\\
                                      \end{array}
                                    \right)({\mathbb Z}^{2d})$.
                                    \end{lemma}

                                    \begin{proof}
                                     Let $\{{\bf e_i}: \ i=1, \cdots, d\}$ denote the standard basis in $\Bbb Z^d$. Associated to this $D$ let
                                      $\Gamma$ and $M$ be given as in Lemma \ref{Gamma-D}.
 Then by Lemma \ref{complete residue classes}, for any $i$ there exists  $z_i\in \Bbb Z^d$ and $\gamma^i\in \{\gamma_1, \cdots, \gamma_n\}$ such that ${\bf e_i}= M^tz_i + M^tD\gamma^i$.  Put $Z:=[z_1 \cdots z_d]$
 and $X=[\gamma_1\cdots \gamma_d]$.  It is clear that $Z$ and $X$ are integer matrices.
 %be the integer matrix consist of $z_i$ in i-th column and
 %$X$ be the integer matrix consist of $\gamma^i$ in the i-th column.
  A direct calculation shows that $(Z+DX)(\Bbb Z^d)= M^{-t}(\Bbb Z^d)$.

 \medskip

Put $P:=\left(
                                      \begin{array}{cc}
                                     Z& -DM \\
                                        X & M\\
                                      \end{array}
                                    \right)$. Then $P$ is an integer matrix with  $\det P=1$ and   $P(\Bbb Z^{2d}) = \Bbb Z^{2d}$. Recall that
 $DM$ is an integer matrix. So  we can  write
                                    $$\Lambda=
 \left(
                                      \begin{array}{cc}
                                        I & D \\
                                        O & I \\
                                      \end{array}
                                    \right)(\Bbb Z^{2d}) = \left(
                                      \begin{array}{cc}
                                        I & D \\
                                        O & I \\
                                      \end{array}
                                    \right)P(\Bbb Z^{2d}) = \left(
                                      \begin{array}{cc}
                                        M^{-t} & O \\
                                       X &  M\\
                                      \end{array}
                                    \right)(\Bbb Z^{2d}). $$ Now take $E=M$, and we are done.
\end{proof}
\medskip

At this point  we are ready to complete the proof of Theorem \ref{UT}.

\begin{proof}[Proof of Theorem \ref{UT}]
 Let $\Lambda=
 \left(
                                      \begin{array}{cc}
                                        A & D \\
                                        O & B \\
                                      \end{array}
                                    \right)(\Bbb Z^{2d})$ where  $B= A^{-t}$ and $\widetilde{D}: = A^{-1}D$ is rational and symmetric.  By Lemma \ref{lemma 3-prim},  the  associated matrix  can be reduced to a block matrix   of the form $\left(
                                      \begin{array}{cc}
                                        I &  \widetilde{D} \\
                                        O & I \\
                                      \end{array}
                                    \right)$ where $\widetilde{D}$ is   rational and symmetric. Therefore, it is sufficient to prove the theorem for  a lattice of the form
            $\Lambda=
 \left(
                                      \begin{array}{cc}
                                        I &D \\
                                        O & I \\
                                      \end{array}
                                    \right)(\Bbb Z^{2d})$  where $D$  is rational and symmetric.  By Lemma \ref{a technical lemma}, there are integer matrices $E$ and $X$ such that $\Lambda=  \left(
                                      \begin{array}{cc}
                                       E^{-t} & O \\
                                        X & E\\
                                      \end{array}
                                    \right)({\mathbb Z}^{2d})$.    The   rest of the proof is  now a conclusion of   Theorem \ref{lower triangle}  and we have completed the proof.
                \end{proof}

            \medskip

 In \cite{HanW4}, the authors  proved that for  any lattice $\Lambda$ formed by an upper triangular matrix exists a window $g$ such that ${\mathcal G}(g,\Lambda)$ is a Gabor orthonormal basis and such $g$ satisfies $\widehat{g} = \chi_K$, where $K$ is the common fundamental domain for the diagonal block matrices $A$ and $B^{-t}$.  However, their proof does not  provide any  constructive technique   for producing a compactly supported window. The proof of Theorem \ref{UT}    provides a technique to    construct a large class of examples of   sets  $K$ forming a Gabor orthonormal basis with respect to the lattice generated by upper triangular matrices. We explain this next.

\medskip

\begin{proposition}\label{converse of UT} Assume that the lattice $\Lambda$ and matrices $A$, $B$ and $D$ are given as in Theorem \ref{UT} and satisfying
 the hypotheses of the theorem. Then there is a set $K$ such that ${\mathcal G}(|K|^{-1/2}\chi_K, \Lambda)$ is a Gabor orthonormal basis for $L^2(\Bbb R^d)$. 
\end{proposition}

 \begin{proof}
 Let $A$, $B$ and $D$ be given. Notice that 
  by Lemma \ref{lemma 3-prim} and the hypotheses of the proposition we know that for any given set $K$, the system ${\mathcal G}(|K|^{-1/2}\chi_K, \Lambda)$  is an orthonormal basis  if and only if ${\mathcal G}(|A^{-1}(K)|^{-1/2}\chi_{A^{-1}(K)}, \tilde \Lambda)$  is an orthonormal basis where $\tilde\Lambda= \left(\begin{array}{cc} I & A^{-1}D\\ O & I\end{array}\right)(\Bbb Z^{2d})$.  And, by Lemma \ref{a technical lemma} we also know that for $\tilde\Lambda$  there   are integer matrices  $X$ and $E$    such
  $$E(\Bbb Z^d) = \{n\in \Bbb Z^d: ~ A^{-1}Dn\in \Bbb Z^d\}$$ 
  and 
 $\widetilde\Lambda=    \left(\begin{array}{cc}  {E}^{-t} & O\\  X &   E \end{array}\right)(\Bbb Z^{2d})$.
  By appealing to Lemma \ref{lemma 3-prim} one more time,
  ${\mathcal G}(|K|^{-1/2}\chi_K, \Lambda)$ is an orthonormal basis if and only if
 $${\mathcal G}\left(|K|^{-1/2}\chi_{K},  \left(\begin{array}{cc} A{E}^{-t} & O\\   A^{-t} X & A^{-t} E \end{array}\right)(\Bbb Z^{2d})\right)$$
  is an orthonormal basis. Now take $K$ as a fundamental domain of  $AE^{-t}(\Bbb Z^d)$ and we are done. 
 \end{proof}

\medskip

The following gives a more  explicit example.

\begin{example}\label{D not I} Let $A = \left(
                                      \begin{array}{cc}
                                        \frac{1}{2} & 0\\
                                   0 & 2 \\
                                      \end{array}
                                    \right)$, $B = \left(
                                      \begin{array}{cc}
                                      2 & 0 \\
                                   0 &   \frac{1}{2} \\
                                      \end{array}
                                    \right)$. Take $D= \left(
                                      \begin{array}{cc}
                                       1 & 0\\
                                   0 &  \frac{1}{3} \\
                                      \end{array}
                                    \right)$.   Then    $A^{-1}D$  is   rational  and symmetric.  Let $\Gamma= \{n\in \Bbb Z^2: \ A^{-1}Dn\in \Bbb Z^2\}$.   $\Gamma$ is a full lattice and a simple  calculation shows that $\Gamma= E(\Bbb Z^2)$ where  $E=     \left(
                                   \begin{array}{cc}
                                1& 0\\
                              0 & 6 \\
                                       \end{array}
                                       \right)$.
                                      Now let $K$   be any   fundamental domain of the lattice $AE^{-t} ({\mathbb Z}^2)= \left(
                                        \begin{array}{cc}
                                1/2& 0\\
                              0 & 1/3 \\
                                       \end{array}
                                       \right)(\Bbb Z^2)$. Then by Proposition \ref{converse of UT}  the system
 $\mathcal G(|K|^{-1/2}\chi_K, \Lambda)$ is a Gabor orthonormal basis for $L^2(\Bbb R^2)$   for the lower triangular lattice $\Lambda$ with diagonal matrices  $AE^{-t}$ and $A^{-t}E$ and any matrix $C$.
                                       \end{example}

\iffalse
 \begin{example}\label{D=I} Let $A = \left(
                                      \begin{array}{cc}
                                        \frac{1}{2} & 0\\
                                   0 & 2 \\
                                      \end{array}
                                    \right)$, $B = \left(
                                      \begin{array}{cc}
                                      2 & 0 \\
                                   0 &   \frac{1}{2} \\
                                      \end{array}
                                    \right)$. Take $D=I$.  Then   $A^{-1}D=A^{-1}$ is  rational  and symmetric.  Let $\Gamma= \{n\in \Bbb Z^2: \ A^{-1}n\in \Bbb Z^2\}$.   $\Gamma$ is a full lattice and we can see  that $\Gamma= E(\Bbb Z^2)$ where  $E=     \left(
                                   \begin{array}{cc}
                                1& 0\\
                              0 & 2 \\
                                       \end{array}
                                  \right)$.
 Now  take $K=[0,1/2]\times [0,1]$. Then $K$ is  a fundamental domain for the lattice $AE^{-t} ({\mathbb Z}^2)$
and $\mathcal G(|K|^{-1/2}\chi_K, \Lambda)$ is a Gabor orthonormal basis for $L^2(\Bbb R^2)$ by Proposition \ref{converse of UT}.
 \end{example}
\fi

% {\bf Remark:} Notice the converse of Theorem \ref{UT}  also holds when $A^{-1}D$ is an integer matrix.
 \section{Examples}\label{Examples}
Consider $\Lambda=L(\Bbb Z^{2d})$ where $L= \left(\begin{array}{cc}
                                        A & O \\
                                        C & B \\
                                      \end{array}
                                    \right)$ is  rational.  In Theorem \ref{lower triangle} we proved that for a set $K$ if $\mathcal G(|K|^{-1/2}\chi_K, \Lambda)$ is an orthonormal basis for $L^2(\Bbb R^d)$ and    $A^tB$ is an integer matrix, then  $K$ is a fundamental domain of $A(\Bbb Z^d)$.  By Theorem \ref{Th_union of FD} this means $N=1$.  Thus  $K$ tiles by $A({\mathbb Z}^d)$ and  is  spectral with spectrum $B^{-t}(\mathbb Z^d)$.  In this section   we will provide examples showing that $N>1$ can also happen in Theorem \ref{Th_union of FD} if $A^tB$ is not an integer matrix and yet the system $\mathcal G(|K|^{-1/2}\chi_K, \Lambda)$ is an orthonormal basis for $L^2(\Bbb R^d)$.

                                       % common fundamental domain ---> window
                                       %  window .--> union of fundamental domains
                                       %  Question: union of fundamental domains ---> window?
                                       %
                                       %
                   %  \begin{enumerate}
% \item[(a)] If ${\mathcal G}(g,\Lambda)$ is a mutually orthogonal set of $L^2({\mathbb R}^d)$, then for any orthogonal packing region  $D$ for $g$,  $D+\Lambda$ is a packing of $\Bbb R^{2d}$.
 %\item[(b)]  Suppose that $f,g\in L^1({\mathbb R}^d)$ are non-negative functions such that $\int f(x) dx = \int g(x) dx= 1$. Suppose that for positive Borel measure $\mu$  on $\Bbb R^d$
%$$
%f\ast\mu \le 1  \ \mbox{and} \ g\ast\mu\le 1.
%$$
%Then $f\ast\mu = 1$  if and only if $ g\ast \mu= 1$.
 %\item[(c)] If ${\mathcal G}(g,\Lambda)$ is a mutually orthogonal set in $L^2({\mathbb R}^d)$ and $D$ is a tight orthogonal packing region for $g$. Then ${\mathcal G}(g,\Lambda)$ is complete  in $L^2(\Bbb R^d)$  if and only if  $D+\Lambda$ is a tiling for $\Bbb R^{2d}$.
%\end{enumerate}

                                      \medskip

                                    We are now ready to present our example of  a  set $K$ which is the union of fundamental domains of  lattice  $A(\Bbb Z^d)$ and  the union of fundamental domains of lattice  $B^{-t}(\Bbb Z^d)$,    $\chi_K$ is a  window function for a possible Gabor orthonormal basis, and $A^tB$ is {\it not an integer matrix}.

\begin{example}\label{mutli-tile K}

Let $K = [0,2]\times [0,1]$ and
$\Lambda = \left(
                                      \begin{array}{cccc}
                                        I& O \\
                                       C & B \\
                                      \end{array}
                                    \right)({\mathbb Z}^{4})$
where
                                    $B=
  \left(
                                     \begin{array}{cc}
                                        1/2 & 0  \\
                                         0&  2
                                      \end{array}
                                    \right)$ and $C=  \left(
                                     \begin{array}{cc}
                                        c_{11}& c_{12}  \\
                                         c_{21}&  c_{22}
                                      \end{array}
                                    \right)$.    The system $\mathcal G(|K|^{-1/2}\chi_K, \Lambda)$ is a Gabor orthonormal basis  if
                         $c_{21}$ is  an odd number.
                                   \end{example}

                                   \begin{proof}
 Observe  that
\begin{align}\notag
[0,2]\times [0,1] &= [0,1]^2 + \{(0,0), (1,0)\} \\\notag
 &= [0,2]\times[0,1/2]+ \{(0,0), (0,1/2)\}.
 \end{align}

% where $ [0,1]^2 $ and $ ([0,1]^2+(1,0)^t) $ are fundamental domains of $\Bbb Z^2$ and $[0,2]\times[0,1/2]$ and  $([0,2]\times[0,1/2])+(0,1/2)^t$ are fundamental domains of $B^{-t}(\Bbb Z^2)$.
This shows that   $K$ is   union of two fundamental domains of $\Bbb Z^2$ and union of two fundamental domains of $B^{-t}(\Bbb Z^2)$, respectively.

  Let $({m}, n) \in {\mathbb Z}^4$ with ${m},{ n}  \in{\mathbb Z}^2$ and $({m},{ n})\neq (0,0)$.  Put
\begin{align}\notag
\mathcal I: = &\int \chi_{K}(x-{m})\chi_{K}(x) e^{-2\pi i \langle C({m}), x\rangle} e^{-2\pi i \langle B({n}), x\rangle}dx\\\notag
=&\int \chi_{(K+{\bf m})\cap K}(x) e^{-2\pi i \langle C({ m}), x\rangle} e^{-2\pi i \langle B({n}), x\rangle}dx .
\end{align}
 Mutual orthogonality of the Gabor system will follow if we can show that $\mathcal I=0$ for any $( m, n)\neq (0,0)$. Note that
$$
(K+m)\cap K \simeq \left\{
                                         \begin{array}{ll}
                                           K, & \hbox{${m} = 0$;} \\
                                                  {[0,1]}^2, & \hbox{${m} = (\pm 1,0)$;} \\
                                           \emptyset , & \hbox{otherwise.}
                                         \end{array}
                                       \right.
$$

(Here, we write $A\simeq B$ if $A = B+k$ for some $k\in{\mathbb Z}^d$.)

  Since $K$ is a union of fundamental domains for $B^{-t}(\Bbb Z^2)$, then  if ${m}=0$, we automatically have   $\mathcal I=0. $

                                   If ${ m}= (\pm 1,0)$,  for any ${ n} = (n_{1},n_{2})$, $\mathcal I$ is equal to the following integral up to a unimodular constant:
\begin{equation}\label{eq6.3}
\begin{aligned}
\int_{[0,1]^2} e^{-2\pi i \langle C((\pm 1, 0)^t), x\rangle}  e^{-2\pi i \langle B({n}), x\rangle}dx=&\int_{[0,1]}\int_{[0,1]} e^{-2\pi i (\pm c_{11}x_1\pm c_{21}x_2)}  e^{-2\pi i (\frac{1}{2}n_{1}x_1+2n_{2}x_2)}dx_1dx_2 \\
=& \widehat{\chi_{[0,1]}}(\pm c_{11}+\frac{1}{2} n_1) \widehat{\chi_{[0,1]}}(\pm c_{21}+2 n_2).
\end{aligned}
\end{equation}
                             It is obvious that the last line   equals to zero for all $(n_1, n_2)$ only  if  $c_{21}$ is an odd number. For other cases of $m$ it is trivial that $\mathcal I=0$. Thus the orthogonality is obtained and
                                the completeness
                              is a direct conclusion of  Proposition \ref{complt}.
\end{proof}

We notice that the previous example exploited the fact that $B({\mathbb Z}^2)$ is exponentially incomplete  for $L^2[0,1]^2$ in (\ref{eq6.3}). The following   example  illustrates a case where, contrary to the previous example,    the  finite union of fundamental domains  cannot form  a Gabor orthonormal basis for any choice of matrix  $C$.

\medskip

                  \begin{example}\label{example2}
                  Let $K=[0,2]^2$ and let  $B$ be the matrix as in Example \ref{mutli-tile K}. Then $K$ is   union of $4$ fundamental domains of  the  lattice  $\Bbb Z^2$ and union of $4$ fundamental domains of  $B^{-t}(\Bbb Z^2)$. However,  there is  no  matrix $C$ for which $\chi_K$ is a   window function yielding an orthonormal basis for the lower triangular  lattice $\Lambda=\left(\begin{array}{cc} I & O\\ C & B\end{array}\right)(\Bbb Z^{4})$.    \end{example}

\begin{proof}
The fact that $K$  is union of  fundamental domains of $\Bbb Z^d$ and $B^{-t}(\Bbb Z^d)$ is straight forward. In short,
$$
\begin{aligned}
K =& [0,1]^2  + \{(0,0), (1,0), (0,1), (1,1)\}\\
  =& [0,2]\times [0,1/2] + \{(0,0), (0,1/2), (0,1), (0,3/2)\}.
\end{aligned}
$$
Suppose that there exists a matrix $C$ such that $\chi_K$ is a   window function for the lower triangular  lattice $\Lambda=\left(\begin{array}{cc} I & O\\ C & B\end{array}\right)(\Bbb Z^{4})$. Thus,    for any $(0,0)\neq (m, n)\in \Bbb Z^{4}$ we must have
$$
 0 = {\mathcal I}:  =\int \chi_{(K+{m})\cap K}(x) e^{-2\pi i \langle C({ m}), x\rangle} e^{-2\pi i \langle B({n}), x\rangle}dx.
$$
From the other side, the only non-empty intersection sets $(K+ m)\cap K$ are
$$
(K+m)\cap K \simeq \left\{
  \begin{array}{ll}
    K, & \hbox{if }{m}=0 \\
    {[0,1]\times[0,2]}, & \hbox{if } {m} = (\pm1,0) \\
    {[0,2]\times[0,1]}, & \hbox{if }  m = (0,\pm1) \\
    {[0,1]}^2, & \hbox{if } m = (\pm1,\pm1).
  \end{array}
\right.
$$
From the last three intersections, we obtain that  if ${\mathcal I}=0$, then  the following three equations must hold for all integer vectors $(n_1,n_2)$, respectively:

  \begin{align}\notag
    \widehat{\chi_{[0,1]}}(\pm c_{11}+\frac{1}{2} n_1) \widehat{\chi_{[0,2]}}(\pm c_{21}+2 n_2) &=0  \\\notag
         \widehat{\chi_{[0,2]}}(\pm c_{12}+\frac{1}{2}n_1) \widehat{\chi_{[0,1]}}(\pm c_{22}+2 n_2)
         &=0\\\notag
       \widehat{\chi_{[0,1]}}(\pm c_{11}\pm c_{12}+ \frac{1}{2} n_1) \widehat{\chi_{[0,1]}}(\pm c_{21}\pm c_{22}+2 n_2) &=0
 \end{align}

\medskip

We claim that if   the first equation holds, then $\widehat{\chi_{[0,2]}}(\pm c_{21}+2 n_2) =0$ for all integers $n_2$ and $c_{21}\in 2{\mathbb Z}+\{\frac{1}{2} ,1,\frac{3}{2} \}$. % But this is straight forward since the zero set for $\widehat{\chi_{[0,2]}}$ is contained in $1/2\Bbb Z$.

\medskip

To justify the claim, suppose that there exists an integer $n_2$ such that $\widehat{\chi_{[0,2]}}(\pm c_{21}+2 n_2) \ne0$. Then $\widehat{\chi_{[0,1]}}(\pm c_{11}+\frac{1}{2} n_1)=0$ for all integers $n_1$. However, this would imply the existence of an exponentials $e^{-2\pi i c_{11} x}$ such that it is orthogonal to all $\{e^{2\pi i (\frac{1}{2}  n) x}: n\in{\mathbb Z}\}$ in $L^2[0,1]$. This is impossible since the exponentials set $\{e^{2\pi i (\frac{1}{2}  n) x}: n\in{\mathbb Z}\}$ is complete in $L^2[0,1]$. Hence, we have only $\widehat{\chi_{[0,2]}}(\pm c_{21}+2 n_2) =0$ for all integers $n_2$. Finally, since the zero set for $\widehat{\chi_{[0,2]}}$ is $\frac{1}{2} {\mathbb Z}$ except zero, then we must have  $c_{21}\in 2{\mathbb Z}+\{\frac{1}{2} ,1, \frac{3}{2} \}$, as desired.

\medskip

Similarly, the second equation  implies that $\widehat{\chi_{[0,1]}}(\pm c_{22}+2 n_2) =0$ for all integers $n_2$ and thus $c_{22}\in 2{\mathbb Z}+1$ must be an odd integer.

\medskip

The third equation implies that
\begin{equation}\label{eq5.4}
\widehat{\chi_{[0,1]}}(\pm c_{21}\pm c_{22}+2 n_2) =0
\end{equation}
for any integers $n_2$. Now, we write $c_{21} = 2k_1+ \frac{1}{2} j$, for some $j = 1,2,3$ and $c_{22} = 2k_2+1$. Then $\pm c_{21}\pm c_{22}\in 2{\mathbb Z}+ \{\frac{3}{2},2,\frac{5}{2}\}$. In the case of fraction, (\ref{eq5.4}) cannot be zero. If $\pm c_{21}\pm c_{22} = 2m+2$, we take $n_2 = -m-1$. Then  (\ref{eq5.4}) will  imply $\widehat{\chi_{[0,1]}}(0) = 1$, which  is impossible.  Thus, the third equation can never be zero. This implies that  such $C$ does not exist.
\end{proof}

\medskip
The following example  proves that the boundedness property for the set $K$ is necessary in Proposition \ref{prop2}. 

                  \begin{example}\label{example3} Let $I_0 = [0,1)$ and $I_n = \left[1-\frac{1}{2^{n-1}},1-\frac{1}{2^{n}}\right)$ for $n\ge1$. Define 
                  $$
                  K = \bigcup_{k\in{\mathbb Z}} (k+ I_{|k|}) 
                  $$
                  The set $K$ is unbounded and we have the following. 
                  
\begin{enumerate}\item $K$ multi-tiles ${\mathbb R}$ by ${\mathbb Z}$ at level 3. However, for any $m\ne 0$, $K\cap( K+m)$ is not a packing of ${\mathbb R}$. Therefore, Proposition \ref{prop2} does not hold if $K$ is unbounded.
 \item Nonetheless, $K$ cannot form a Gabor orthonormal basis using lattices of the form $
  \left(
                                     \begin{array}{cc}
                                        1 & 0  \\
                                         c&  1
                                      \end{array}
                                    \right) ({\mathbb Z}^2)$.  \end{enumerate}
\end{example}

\begin{proof}
The fact that $K$ mult-tiles ${\mathbb R}$ by ${\mathbb Z}$ at level 3 follows from a direct observation, so  we omit the details here. We note that for any $m\ne 0$,
$$
K\cap (K+m)\supset I_{|m|} \cup (I_{|m|}+m).
$$
Hence, for all $x\in I_{|m|}$, $\sum_{n\in{\mathbb Z}} \chi_{K\cap(K+m)}(x+n)\ge 2$. Therefore, it is never a packing for any $m\ne 0$. To show  the last statement, notice 

$$
K\cap (K+1) = \left[0,\frac12\right)\cup\left[1,\frac32\right), \ K\cap (K+2) = \left[\frac12,\frac34\right)\cup\left[1,\frac32\right)\cup\left[\frac52,\frac{11}{4}\right).
$$
Suppose that $K$ forms a Gabor orthonormal basis using some lattice of the form $
  \left(
                                     \begin{array}{cc}
                                        1 & 0  \\
                                         c&  1
                                      \end{array}
                                    \right) ({\mathbb Z}^2)$. Then $\widehat{\chi_{K\cap(K+1)}}(c+m) = 0$ and $\widehat{\chi_{K\cap(K+2)}}(c+m) = 0$ for all $m\in{\mathbb Z}$. In particular, 
 $$
 \widehat{\chi_{K\cap(K+1)}} (c+m) = (1+e^{2\pi i (c+m)})\widehat{\chi_{[0,1/2]}}(c+m) =0, \forall m\in{\mathbb Z}.
                                    $$
We see that the only possibility of the above is that $c\in\frac12+{\mathbb Z}$. We now consider 
$$
\widehat{\chi_{K\cap(K+2)}} (\xi) = e^{2\pi i \frac12\xi}\left(1+e^{2\pi i \frac12\xi}+e^{2\pi i \frac34\xi}+e^{2\pi i \frac2\xi}\right)\widehat{\chi_{[0,1/4)}}(\xi).
$$
If $c \in\frac12+{\mathbb Z}$ and $\widehat{\chi_{K\cap(K+2)}} (c) =0$, we must have 
$$
0=1+e^{2\pi i \frac12c}+e^{2\pi i \frac34c}+e^{2\pi i 2c} = 2+e^{2\pi i \frac12c}+e^{2\pi i \frac34c}.
$$ 
This forces $e^{2\pi i \frac12c} = e^{2\pi i \frac34c}=-1$. Thus, 
$$c\in 2(1/2+{\mathbb Z})\cap \frac43(1/2+{\mathbb Z}) = (1+2{\mathbb Z})\cap \left(\frac23+\frac43{\mathbb Z}\right)\subset 1+2{\mathbb Z}$$
This is a contradiction since $c\in\frac12+{\mathbb Z}$ is never an  integer. This completes the proof.
 \end{proof}

\section{Discussions and Open problems}\label{Open problems}

This paper investigates the Fuglede-Gabor Problem \ref{our conjecture1} over the lattices. We believe that this  problem should be true for all lattices. We solved the problem completely in dimension one when the lattice  is rational and in higher dimensions when the lattice is integer.   In what follows we shall explain how to resolve Problem \ref{our conjecture1}  in full generality for any  lattices  $\Lambda = \left(
                                      \begin{array}{cc}
                                       A  & D \\
                                        C & B \\
                                      \end{array}
                                    \right)({\mathbb Z}^{2d})$. In fact,
 it is sufficient  to solve the following two cases.
\begin{enumerate}
\item[(1)]\label{item 1} {\bf Rational case:} After converting a rational  lattice into a lower triangular rational matrix by Corollary \ref{M=N} and   reducing the  matrix   where $A=I$  by    Lemma \ref{lemma 3-prim},   $\Lambda$ is a lattice of the form of   $  \left(
                                      \begin{array}{cc}
                                        I  & O \\
                                        C & B \\
                                      \end{array}
                                    \right)({\mathbb Z}^{2d})$, where $B$ and $C$ are $d\times d$ rational matrices but  $B$ is not necessarily an integral matrix.

                                    \medskip

\item[(2)] {\bf Irrational case:}  After a reduction process  by Lemma \ref{lemma 3-prim}, $\Lambda$ is a lattice of the form of  
$$\left(
                    \begin{array}{cc}
                       I & D\\
                       C & B  \\
                    \end{array}
                  \right) ({\mathbb Z}^{2d})$$  which   contains  irrational entries.
\end{enumerate}

\medskip
In what follows, we shall discuss these two cases in more details.

\subsection{Rational Case}
Let  $\Lambda= \left(
\begin{array}{cc}
A & O\\
C & B
\end{array}\right)(\Bbb Z^{2d})$ be a lower triangular rational matrix.
 Example \ref{mutli-tile K}  tells us that when $A^tB$ is  a non-integer matrix, $K$ is  not a  fundamental domain for the lattices  $A(\Bbb Z^d)$ and $B^{-t}(\Bbb Z^d)$ but the union of their fundamental domains. However, in Example \ref{example2} we see that there exists no $C$ such that $K$, as the  union of fundamental domains, is a set whose characteristic function generates a  Gabor orthonormal basis associated to the given matrices $A$ and $B$. We predict that this failure is due to the  number of decompositions of $K$ into fundamental domains of $B^{-t}(\Bbb Z^d)$.
 In this concern and in relation to  the examples illustrated  in Section \ref{Examples},
  we conjecture the following problem  for the lattices
$\Lambda= \left(
                    \begin{array}{cc}
                       I & D\\
                       C & B  \\
                    \end{array}
                  \right) ({\mathbb Z}^{2d})$.
 \medskip

  \medskip

  \begin{conjecture}\label{our conjecture}   Let $K\subset \Bbb R^d$ and $B$ be a  rational matrix with  $\det(B)=1$. Let
  $s$ be  the least common multiple of the denominators of the  matrix entries  $(b_{ij})$.
 There  is a matrix $C$ such that $\mathcal G(|K|^{-1/2}\chi_K, \Lambda)$ is a  Gabor orthonormal  basis  for $L^2(\Bbb R^d)$  if and only if $K$ tiles   and
  $$K= \bigcup_{i=1}^sE_i= \bigcup_{i=1}^sF_i $$
  where $E_i$  and $F_i$ are
   fundamental  domains  of $\Bbb Z^d$   and  $B^{-t}(\Bbb Z^d)$, respectively.
 \end{conjecture}

It is obvious that the conjecture automatically holds when $B$ is an  integer matrix which means  $s=1$.  Indeed,  this is  the result of Theorem \ref{lower triangle}.

%                                    \begin{problem} Assume that $\mathcal G(|K|^{-1/2}\chi_K, \Lambda)$ is a Gabor orthonormal basis with respect to lattice $\Lambda$. If $K$ does not tile, then
%                                    there exists a
%                                      non-zero $m\in \Bbb Z^d$ such that $K\cap K+m$ is packing for the lattice $B^{-t}(\Bbb Z^d)$.
%                                     \end{problem}
%
%                                      Given that the answer to  the problem is positive and if  $K$ does not tile there exists  such $m\neq 0$, then Lemma \ref{lemma 4} will conclude  that the set $\{e^{2\pi i \langle B(n), x\rangle}: \ n\in \Bbb Z^d\}$ is complete in $L^2(K\cap K+m)$.
%                                      From the other side, by the orthogonality condition we have
%                                      $$\int_{K\cap K+m} e^{2\pi i \langle C(m),x\rangle} e^{2\pi i \langle B(n), x\rangle} dx = 0 \quad \forall \ n\in\Bbb Z^d .$$
%                                      Since $C$ is a non-zero matrix,   the completeness of the exponentials $\{e^{2\pi i \langle B(n), x\rangle}\}$ implies that  $e^{2\pi i \langle C(m),x\rangle} = 0$, for a.e. $x$, which is impossible.
%

 \medskip

 Observe that if  $K$ is as  in Conjecture \ref{our conjecture},   then for any non-zero $n\in \Bbb Z^d$ we have $\int_K e^{2\pi i \langle Bn, x\rangle} dx=0$. And, there are only finitely many $m\in \Bbb Z^d$ such that $|K\cap (K+m)|\neq 0$, as $K$ is a finite union of fundamental domains of $\Bbb Z^d$.  To prove the orthogonality,  one must first show the existence of  a matrix $C$ such that for  all   $n\in \Bbb Z^d$
\begin{equation}\label{eq6}
\int_{K\cap K+m}  e^{-2\pi i \langle Cm, x\rangle}  e^{-2\pi i \langle Bn, x\rangle}  dx=0,
\end{equation}
which is equivalent to the exponential incompleteness of the lattices $B(\Bbb Z^d)$ for $L^2(K\cap K+m)$. It appears that we need to study the exponential completeness of the lattices over different domains.  In fact, the following problem has not yet had a definite answer.
  
  \medskip

\begin{problem} Given a set $K$ with positive and finite measure,  classify all invertible $d\times d$ matrices $B$ for which 
 the exponents $\{ e^{-2\pi i \langle Bn, x\rangle} :  n\in \Bbb Z^d\}$ are  exponentially complete in   $L^2(K)$. 
\end{problem}

\medskip

Proposition \ref{prop_G} provides a sufficient condition for matrices $B$  when $K=[0,1]^d$. Unfortunately, the converse of the proposition does not hold in dimensions $d=2$ and higher, as the counterexample following Proposition \ref{prop_G} shows. However,  in Proposition \ref{dimension one} we obtain a full characterization of exponential completeness for $[0,1]$ in dimension $d=1$ for integer lattices in $\Bbb Z$.

\medskip 

As mentioned earlier, the exponential completeness of a set does not imply the completeness of the set in general. For a recent developments in study of 
 the completeness, frame and Riesz bases properties of  exponentials we refer the reader to the paper by  De Carli and her co-authors \cite{De Carli}.

%Similar problems can also be asked for exponential completeness replaced by completeness, frame or Riesz basis. We are glad that De Carli  et al recently solved it \cite{De Carli}.   

%
%\begin{problem}
%For what lattice $B(\Bbb Z^d)$ it will be complete/incomplete/ONB for $L^2[0,1]^d$? Classify it.
 %\end{problem}

\subsection{Irrational Cases} The case for irrational lattices is more challenging and complicated. It appears that the lower and upper triangular case is asymmetric.  We have seen that in Theorem \ref{lower triangle}, the lower triangular block matrix $C$ is not involved in the statement.  Thus, irrational entries are allowed for lower triangular matrices. On the other hand, Han and Wang \cite{HanW4} implicitly conjectured the following problem in their paper  \cite{HanW1}:

\medskip

 {\bf Han and Wang\rq{}s Conjecture:} {\it  Let $\Lambda = \left(
                                      \begin{array}{cc}
                                        1 & \alpha \\
                                        0 & 1 \\
                                      \end{array}
                                    \right)({\mathbb Z}^{2})$ where $\alpha$ is irrational. Then there doesn't exist compactly supported window $g$ such that ${\mathcal G}(g,\Lambda)$ forms a Gabor orthonormal basis for $L^2({\mathbb R})$.}

 \medskip

 Observe that if $\Lambda=  \left(
                                      \begin{array}{cc}
                                     1& \alpha \\
                                   0 &   1 \\
                                      \end{array}
                                    \right)$, with $\alpha$ irrational, our method applied to construct Example \ref{D not I}  would not work for the existence of $K$ since in that case $\Gamma=\{0\}$. This observation predicts that  Han and Wang Conjecture  \cite{HanW1}  might be true, although we do not have a proof for it now.  However, a simple calculation shows that the function $\chi_{[0,1]}$ can not be a window function for this  lattice. 

                                    \medskip

  %In \cite{L01} and \cite{L02}, \L aba established the conjecture for the union of two intervals on the real line and provided a series of
%%connections between the study of orthonormal bases and interesting problems in algebraic number theory.
%  Yet the conjecture may still be true in several important special cases.
%  In spite of  the disproof of its general validity, the conjecture has  still generated many interesting results.  Most recently, it was proved by the second listed author and collaborators that the Fuglede Conjecture holds in some abelian finite groups (\cite{IMP}). For the link between tiling property and existence of general bases in locally compact abelian groups see \cite{BHM16}.

\subsection{Full generality of Fuglede-Gabor Problem \ref{our conjecture1}} 
It is known that non-symmetric convex bodies   as well as  convex sets with a point of non-vanishing Gaussian curvature    have no basis of exponentials and yet they do not tile \cite{K99,IKT01}. Recently, similar results were also   proved for Gabor bases.  Indeed, the authors in \cite{IM17} proved that in dimensions $d\neq1$    (mod 4),  convex sets with a point of non-vanishing Gaussian curvature cannot generate any Gabor orthonormal basis with respect to any countable time-frequency set $\Lambda$. Also, the authors in \cite{CL} proved that non-symmetric convex polytopes do not produce any Gabor orthonormal basis with respect to any  countable time-frequency set. However, the result for non-symmetric  convex domains is not known yet.  The existent  results predict that  Fuglede-Gabor problem will  still be true to some  extent.

\medskip

%{\color{blue} I think we should keep it. People know (or even believe) FG problem can be wrong. Some people may think the counter-example can be obtained trivially. I want to emphasize here the delicacy to prove or disprove it. This paragraph says using the counterexample of Fuglede conjecture cannot produce an immediate example. }

%{\color{red}  (I would remove this paragraph, since we are telling that F-G problem can be wrong though. Also a contradiction to the following paragraph as well as the following example.)}
 One may notice that there are also 
  examples of  spectral sets which do not tile by translations (see e.g.
 \cite{T04}, \cite{Matolcsi} and \cite{KM06}). Therefore, by a result of Dutkay and the first listed author \cite{Dutkay-Lai}, it is known that the indicator function of these sets can not serve as    Gabor orthonormal window with resepct to  any separable time-frequency set. However, they may still produce a  Gabor orthonormal basis using some non-separable and countable  time-frequency sets. 
  This   requires some input of new ideas.

\medskip

Finally, it is known that octagon does not tile ${\mathbb R}^2$ by translations, but it is a multi-tile by ${\mathbb Z}^2$. The following example tells us that it does not form Gabor orthonormal basis using any lattices, confirming that the Fuglede-Gabor problem  holds up to some extent.

\medskip

                                      \begin{example}\label{Octagon} Let ${\mathcal O}_8\subset \Bbb R^2$ be the octagon symmetrically centred at the origin with integer vertices $\{(\pm 1, \pm 2), (\pm 2, \pm 1)\}$. ${\mathcal O}_8$ multi-tiles $\Bbb R^2$ with $\Bbb Z^2$  and it is the  union of  $s=14$ fundamental domains of $\Bbb Z^2$.  Yet there doesn't exist any $\Lambda=   \left(
                                      \begin{array}{cc}
                                        I & O \\
                                        C & B \\
                                      \end{array}\right) (\Bbb Z^4)$ with rational matrix $B$  for which   $\mathcal G(|{\mathcal O}_8|^{-1/2}\chi_{{\mathcal O}_8}, \Lambda)$ forms an  orthonormal basis for $L^2(\Bbb R^2)$ .
                                      \end{example}

                                      The proof of this example is involved and so it will be provided in the next section.

\section{Gabor orthonormal bases on Octagon and Proof of Example \ref{Octagon}}\label{Appendix Octagon}

In this appendix, we will prove that the octagon symmetrically centered at the origin with integer vertices $\{(\pm 1, \pm 2), (\pm 2, \pm 1)\}$ cannot admit any Gabor orthonormal basis  with  $\Lambda=   \left(
                                      \begin{array}{cc}
                                        I & O \\
                                        C & B \\
                                      \end{array}\right) (\Bbb Z^4)$, where  $B$ is  a rational matrix. For this we need the following lemma.

\begin{lemma}\label{lemma_shear}
Let $M = \left(
                                      \begin{array}{cc}
                                        \alpha & 0 \\
                                        0 & 1/\alpha \\
                                      \end{array}\right)$ and $M_{\beta} = \left(
                                      \begin{array}{cc}
                                        \alpha & 0 \\
                                        \beta & 1/\alpha \\
                                      \end{array}\right)$, $\alpha\neq 0$. Then $\Omega$ is a multi-tile by $M_{\beta}({\mathbb Z}^2)$ if and only if  $\Omega$ is a multi-tile by $M({\mathbb Z}^2)$.
\end{lemma}

\begin{proof}
Let $\alpha>0$ and put $Q_{\alpha} = [0,\alpha)\times[0,1/\alpha)$. Then $Q_{\alpha}$ is a fundamental domain for the lattice $M(\Bbb Z^2)$. We claim that $Q_{\alpha}$ is a fundamental domain for $M_{\beta}(\Bbb Z^2)$ for any $\beta\in\Bbb R$. Indeed, this follows from a direct calculation: For almost every $(x,y)$, we have 
$$
\begin{aligned}
\sum_{m,n\in{\mathbb Z}}\chi_{Q_{\alpha}}(x-\alpha m,y-\beta m-n/\alpha) =& \sum_{m\in{\mathbb Z}}\chi_{[0,\alpha)}(x-\alpha m)\left(\sum_{n\in{\mathbb Z}}\chi_{[0,1/\alpha)}(y-\beta m-n/\alpha)\right) \\
=&\sum_{m\in{\mathbb Z}}\chi_{[0,\alpha)}(x-\alpha m)=1.\\
\end{aligned}
$$

Let $\Omega$ be any set which is a tile by $M({\mathbb Z}^2)$. Let $\{E_{(u,v)} : (u,v)\in \Bbb Z^2\}$ be a  partition of $Q_{\alpha}$  such that 
$$
\Omega = \bigcup_{(u,v)\in{\mathbb Z}^2} (E_{(u,v)}+(\alpha u,v/\alpha)) .$$ 
 
%where 
  %$E_{(u,v)}$ are partition of $Q_{\alpha}$ and  $\bigcup_{(u,v)\in{\mathbb Z}^2} E_{(u,v)} = Q_{\alpha}$.   
  Then set $\Omega$ is a tile by $M_\beta(\Bbb Z^2)$. Indeed,  
$$
\begin{aligned}
\sum_{m,n\in{\mathbb Z}}\chi_{\Omega}(x-\alpha m,y-\beta m-n/\alpha)=&\sum_{m,n\in{\mathbb Z}} \sum_{(u,v)\in{\mathbb Z}^2}\chi_{E_{(u,v)+(\alpha u,v/\alpha)}} (x-\alpha m,y-\beta m-n/\alpha)\\
=&\sum_{(u,v)\in{\mathbb Z}^2}\sum_{m,n\in{\mathbb Z}} \chi_{E_{(u,v)}} (x-\alpha (m+u),y-\beta m-(n+v)/\alpha)\\
=&\sum_{(u,v)\in{\mathbb Z}^2}\sum_{m,n\in{\mathbb Z}} \chi_{E_{(u,v)}} (x-\alpha m,y-\beta m-n/\alpha)\\
=&\sum_{m,n\in{\mathbb Z}}\sum_{(u,v)\in{\mathbb Z}^2} \chi_{E_{(u,v)}} (x-\alpha m,y-\beta m-n/\alpha)\\
=&\sum_{m,n\in{\mathbb Z}} \chi_{Q_{\alpha}} (x-\alpha m,y-\beta m-n/\alpha)=1.\\
\end{aligned}
$$
Note that above we used an application of Fubini\rq{}s theorem.   
Now let 
 $\Omega$ be   a multi-tile with respect to the lattice $M(\Bbb Z^2)$. Then $\Omega = \bigcup_{i=1}^N\Omega_i$, where $\Omega_i$ are tiles by $M({\mathbb Z}^2)$. Thus, by the previous  results  above, each $\Omega_i$ is a tile by $M_\beta(\Bbb Z^2)$, hence we have the following. 
$$ 
\sum_{m,n\in{\mathbb Z}}\chi_{\Omega}(x-\alpha m,y-\beta m-n/\alpha)=\sum_{i=1}^N\sum_{m,n\in{\mathbb Z}}\chi_{\Omega_i}(x-\alpha m,y-\beta m-n/\alpha)=N.
$$

 The converse can be obtained by a similar calculation. 
 
\end{proof}

   \begin{proof}[Proof of Example \ref{Octagon}]
                                        Before we prove our claim, note that  if the family
                                 ${\mathcal G}(|{\mathcal O}_8|^{-1/2}\chi_{{\mathcal O}_8}, \Lambda)$   is a Gabor basis and $B$ is an integer matrix, then according to Theorem \ref{lower triangle}   the set ${\mathcal O}_8$ must  tile which is impossible. Thus, we assume that $B$ has some non-integer rational entries.   Let ${\bf m}_0=(3,2)$. Then  ${\mathcal P}:=K\cap (K+{\bf m}_0)$,   $K=\mathcal O_8$, is the parallelogram   with vertices
                                             $\{(1,2), (2,1), (1,1), (2,0)\}$. Hence,
                                             $$
                                             {\mathcal P} = Q[0,1]^2+(1,1)^t, \ \mbox{where} \  Q = \left(
                                      \begin{array}{cc}
                                        1 & 0 \\
                                        -1 & 1 \\
                                      \end{array}\right).
                                             $$
                                              The  Fourier transform of $\chi_{\mathcal P}$ at $\xi  = (\xi_1,\xi_2)$ is given by 
                                             $$
                                              \widehat{\chi_{{\mathcal P}}}(\xi) = c ~ \widehat{\chi_{[0,1]^2}}(Q^{T}\xi)
                                             $$
                                              where $c:=c(\xi)$ is some unimodular constant.

                                      \medskip

                                   By  the mutual orthogonality of the element $( m_0, n)$ with $(0, 0)$    for all $ n\in{\mathbb Z}^2$, we obtain 
                      
                                              \begin{align}\label{integral-octagon}
                                     0= I:= &   \int_{K\cap (K+{m}_0)} e^{-2\pi i \langle C m_0, x\rangle} e^{-2\pi i \langle Bn,x\rangle} dx  
                                      = c\cdot\widehat{\chi_{[0,1]^2}}(Q^TCm_0+Q^TB n) . 
                                               \end{align}

                                         If $Cm_0=0$, by putting ${n}=0$, then one must have  $c=0$, which is a contradiction . Otherwise, (\ref{integral-octagon})               shows that $Q^TB({\mathbb Z}^2)$ is exponentially incomplete for  $L^2[0,1]^2$. By Theorem \ref{Theorem_Appendix1} in the appendix, $Q^TB$ is equivalent to  one of the following two forms.
                                        \begin{align}\label{matrices forms}
                                        \left(
                                      \begin{array}{cc}
                                        \frac{1}{q'} & 0 \\
                                        r' & q' \\
                                      \end{array}\right) \  \  \mbox{or} \ \  \left(
                                      \begin{array}{cc}
                                        p' & 0 \\
                                        \frac{r''}{s''} & \frac{1}{p'} \\
                                      \end{array}\right)
                              \end{align}
                                      for some integers $p',q', r'>1$, gcd of $r'$ and $q'$ is strictly greater than 1 and  $(r'',s'')$ is relatively prime. 
                                      
                                      \medskip

                                          We now prove that $B$ also is equivalent to  the same desired form in (\ref{matrices forms}). We are going to establish the first case, the second case is similar. Recall that $Q =  \left(
                                      \begin{array}{cc}
                                        1& 0 \\
                                        -1 & 1 \\
                                      \end{array}\right)$. If $Q^TB\tilde U = \left(
                                      \begin{array}{cc}
                                        \frac{1}{q} & 0 \\
                                        r & q \\
                                      \end{array}\right)$ for some unimodular integer matrix $\tilde U$,  then $B\tilde U= \left(
                                      \begin{array}{cc}
                                        \frac{1+rq}{q} & q \\
                                        r & q \\
                                      \end{array}\right)$. Note that $1+rq$ and $q^2$ is relatively prime. Thus there exist co-prime  integers $u,v$ such that $(1+rq)u+q^2v = 1$. Define $ U = \left(
                                      \begin{array}{cc}
                                        u & -q^2 \\
                                        v & 1+rq \\
                                      \end{array}\right)$. Then $\det(U)=1$ and the lattice $B\tilde U(\mathbb Z^2) = B\tilde UU({\mathbb Z}^2)$. Note that 
                                      $$
                                      B\tilde UU = \left(
                                      \begin{array}{cc}
                                        \frac{1+rq}{q} & q \\
                                        r & q \\
                                      \end{array}\right)\left(
                                      \begin{array}{cc}
                                        u & -q^2 \\
                                        v & 1+rq \\
                                      \end{array}\right) = \left(
                                      \begin{array}{cc}
                                        \frac{1}{q} & 0 \\
                                        ru+vq & q\\
                                      \end{array}\right).
                                      $$ 
                                          This shows that $B$ is equivalent to the desired form in (\ref{matrices forms}). 
                              For the rest, we prove that the  neither of these forms   can  form an Gabor orthonormal basis.   By the Lemma \ref{lemma_shear} and Theorem \ref{Th_union of FD}, we just need to prove that the octagon ${\mathcal O}_8$ is not a multi-tile for the matrices $\left(\begin{array}{cc}
                                        p& 0 \\
                                        0& 1/p \\
                                      \end{array}\right)$ and $\left(\begin{array}{cc}
                                        1/p& 0 \\
                                        0 & p \\
                                      \end{array}\right)$ where $p>1$ is an integer. By the symmetry of the octagon, we just need to prove that ${\mathcal O}_8$ is not a multi-tile for the first one.  Let $B_p:=\left(\begin{array}{cc}
                                        p& 0 \\
                                        0& 1/p \\
                                      \end{array}\right)$. To prove this, we first note that an elementary calculation shows that the area of ${\mathcal O}_8$ is 14. If it   multi-tiles by $B_p({\mathbb Z}^2)$, then for almost every  $x\in{\mathbb R}^2$, the cardinality of the set $(x+B_p({\mathbb Z}^2))\cap {\mathcal O}_8$ is $l=14$.
                                      To obtain a contradiction 
                                      consider the rectangle $R_p:=[0,1)\times[0,1/p)$ for $p>1$. It is a simple observation that $R_p$ can be covered  $12p$ by translations of $\mathcal O_8$ by the matrix $B_p$. This is a contradiction to the level of multi-tiling $l=14$,  thus we have completed the proof.

                                      \iffalse 
                                       % We now divide our consideration into two cases.
                                      
                                      \medskip
                                      
                                      If $p=2$, we consider the upper corner of  triangle formed by the vertices $(1,2), (1,3/2)$ and $(3/2,3/2)$ and denote it by $\Delta$, then for any $x\in \Delta$, 
                   $$ 
                   x+\left\{(k,j/2): k\in\{0,-2\},j \in\{0,-1,-2,-3,-4,-5,-6,-7\}\right\}\in (x+B_2({\mathbb Z^2}))\cap {\mathcal O}_8.
$$
There are 16 elements and $\Delta$ has positive measure. This shows that ${\mathcal O}_8$ cannot be a multi-tile by $B_2({\mathbb Z}^2)$.

\medskip

If $p\ge3$, observe that  the rectangle $[0,1)\times[0,1/p)$ is covered $12p$ times by $B_p(\Bbb Z^2)$, and $p\geq 3$. 

   $ \{(0,j/p): -2p\leq  j\leq 2p-1\}$  exactely $4p$ times.  Since $4$ does not divide $14$, thus we have completed the proof.
   \fi 

\end{proof}

 \appendix 
                                      \section{Exponential Completeness}\label{Exponential Completeness}

                                      In this appendix, we will study some special cases for exponential completeness and one of the cases will be used to prove Example \ref{Octagon}. We will focus our attention on $\Lambda$ to be a subgroup of integers. 
                                                               The following proposition provides a   full characterization of exponential completeness for $[0,1]\subset \Bbb R$   for integer lattices.    

                                      \medskip
                                      
                                      \begin{proposition}\label{dimension one} Let 
                                       $\Lambda = p{\mathbb Z}$, $p\neq 0$. Then $\Lambda$ exponentially incomplete for $L^2[0,1]$ if and only if  $p$ is an integer and $|p|>1$.  
                                       \end{proposition}
                                      
                                      \begin{proof} First assume that $p$ is an integer and $|p|>1$. 
                                      Then in Definition \ref{exponential completeness} we take $\xi = 1$ and we will have $\widehat{\chi_{[0,1]}}(1+p n) =0$ for all $n\in{\mathbb Z}$. This shows that such $\Lambda$ is exponentially incomplete.
                                      
                                      \medskip
                                      
                                      Conversely, if $\Lambda$ is exponentially incomplete, then we can find $a\in{\mathbb R}$ such that $\widehat{\chi_{[0,1]}}(a+p n) = 0$ for all $n\in{\mathbb Z}$. This means that $a+p n\in{\mathbb Z}\setminus\{0\}$ for all $n\in{\mathbb Z}$. Putting $n=0$ implies that $a\in{\mathbb Z}\setminus\{0\}$ and putting $n=1$ implies that $p$ is an integer and $p\neq -a$.   Now let $n=a$ and $n=-a$, respectively. Then  $a+ap$ and $a-ap$ are in ${\mathbb Z}\setminus\{0\}$ which imply that  $|p|\neq 1$.  
                                    % Note that $p \ne 1$, otherwise, one can  take $n=-a$ and then $a+p n = 0$ but $\widehat{\chi_{[0,1]}} =1$, and this is a contradiction. 
               Thus, $|p|>1$  and this completes the proof.  
                                      \end{proof}
                                      
                       %          {\color{blue} (I don't know the purpose of this remark. There is no obvious connection of exponential completeness to AP. Also, AP of infinite length can be thought as coset of a subgroup which is contained in the proposition below.)}

                           %         {\it Remark:}   Notice the previous result also holds true since the zeros for $\widehat{\chi_{[0,1]}}$ contains all arithmetic progression of infinite length with initial non-zero integers $a$ and  integer steps $|p|>1$.   

                                      The classification of exponential incomplete lattices   in higher dimensions  is not straight forward as in the case dimension one.  Given a vector ${\bf v} = (v_1,...,v_d)\in{\mathbb R}^d$, we define the ${\mathbb R}$-subgroup  $G({\bf v})$ generated by the entries of ${\bf v}$ as 
                                      $$
                                      G({\bf v}) := \{v_1n_1+...+v_dn_d: (n_1,...,n_d)\in{\mathbb Z}^d\} = \{\langle {\bf v},n\rangle: n\in{\mathbb Z}^d\}.
                                      $$
                                      
                                      \begin{proposition}\label{prop_G}
                                      Let $B = \left(\begin{array}{ccc}
                                        - & {\bf v}_1& - \\
                                        -&\vdots & - \\
                                        - & {\bf v}_d& - \\
                                      \end{array}\right)$, $d>1$. Suppose that there exists ${\bf v}_k$ such that $G({\bf v}_k) = p{\mathbb Z}$ for some integer $p$. If $|p|>1$, then $B({\mathbb Z}^d)$ is exponentially incomplete in $L^2([0,1]^d)$.
                                      \end{proposition}

  \begin{proof}
  We note that for any $\alpha= (\alpha_1,...,\alpha_d)\in{\mathbb R}^d$ and any $n\in{\mathbb Z}^d$, 
  $$
  \widehat{\chi_{[0,1]^d}}(\alpha+Bn) = \prod_{i=1}^d \widehat{\chi_{[0,1]}}(\alpha_i+ \langle {\bf v}_i, n\rangle)
  $$
  Clearly, if one of  $G({\bf v}_k) = p{\mathbb Z}$ for some $p>1$, then we just let $\alpha_k = 1$ and obtain that  $\widehat{\chi_{[0,1]}}(1+ \langle {\bf v}_k, n\rangle)=0$ for all $n\in{\mathbb Z}^d$. This completes the proof.
  \end{proof}

                                     \medskip 
                                     
                                       {\it Remark:} If $B$ is an integral matrix, then 
                                       one can reduce B to a lower triangular  matrix with $b_{11}=p$. This means that    $G(v_1)=p \Bbb Z$, hence the result holds if $|p|>1$. 
 However, the assumption $|p|>1$ is not necessary, thus 
 the converse of the proposition does not hold in general. For this, take 
     $
                               B=\left(
                                      \begin{array}{cc}
                                      p & 0 \\
                                        r & q\\
                                      \end{array}\right)$ where $p$ is a non-zero  integer, $q$ is an integer with $|q|>1$, $(p,q)=1$  and $r$ is a real number.
                                      Take $\alpha=(p,p+r)$. 
                                      Then for  $m=-1$, we have  $\widehat{\chi_{[0,1]}}(p+r(1+m)+qn)=0$ since $p$ and $q$ are co-primes. For $m\neq -1$, we have 
                                          $\widehat{\chi_{[0,1]}}(p(1+m))=0$.   Thus, for all ${\bf n}=(m,n)\in \Bbb Z^d$ we have 
                                     $
  \widehat{\chi_{[0,1]^d}}(\alpha+B{\bf n}) =0$, and the exponential incompleteness holds.

                                               %         {\color{blue} To obtain a counterexample of Proposition A.2,  we should write $B=\left(
                                      %\begin{array}{cc}
                                     % 1 & 0 \\
                                       % r & q\\
                                      %\end{array}\right)$, $q>1$ and gcd($r,q$)=1. Then $G(1,0) = G(r,q) = {\mathbb Z}$, but  take $\alpha = (1,1+r)$, we need to show $1+m\in {\mathbb Z}\setminus\{0\}$ or $(1+r)+rm+qn\in {\mathbb Z}\setminus\{0\}$. Indeed, If $m\ne -1$, $1+m$ is always a non-zero integer, but if $m=-1$, then $1+r-r+qn = 1+qn $ is always a non-zero integer. We shows the exoponential incompleteness}

          \medskip 
          
      The following theorem is a special case in dimension $d=2$  which fits into our needs for the  proof of Example \ref{Octagon}.
      
              \medskip 
  
  \begin{theorem}\label{Theorem_Appendix1}
Let        $B$ be  a $2\times 2$ rational matrix with $\det(B)=1$. Then $B({\mathbb Z}^2)$ is exponentially incomplete for $L^2[0,1]^2$  if and only if $B$ is equivalent to                
$\left(
                                      \begin{array}{cc}
                                        \frac{1}{q'} & 0 \\
                                        r' & q' \\
                                      \end{array}\right)$ or $\left(
                                      \begin{array}{cc}
                                        p' & 0 \\
                                        \frac{r''}{s''} & \frac{1}{p'} \\
                                      \end{array}\right)$ for some integers $p',q', r'>1$ with the greatest common divisor $gcd(r',q')>1$ and  $r''$ and $s''$ are relatively prime. 
  \end{theorem}

                          \begin{proof}
                          The sufficiency follows from Proposition \ref{prop_G} since $G(r',q') =d{\mathbb Z}$ for $d:=gcd(r',q')$ and $G(p',0) = p'{\mathbb Z}$.  We now prove the necessity. By Corollary \ref{M=N}, there exists an integer matrix $U$ of determinant one such that $BU$ is lower triangular.         
                          Since $ \det(B)=1$ and the matrix is rational,  we can  assume that 
                                $$
                               B=\left(
                                      \begin{array}{cc}
                                        \frac{p}{q} & 0 \\
                                        \frac{r}{s} & \frac{q}{p} \\
                                      \end{array}\right)
                                $$            
                                      and the pairs of integers $(p,q)$ and $(r,s)$ are respectively relatively prime. Writing ${\bf x}_0= (x_1,x_2)$, $B({\mathbb Z}^2)$ is exponentially incomplete only if  we have some ${\bf x}_0$ such that
                                      \begin{equation}\label{eq6.1}
  \widehat{\chi_{[0,1]}}\left(x_1+\frac{p}{q}n_1\right) \widehat{\chi_{[0,1]}}\left(x_2+\frac{r}{s}n_1+\frac{q}{p}n_2\right)=0 
                                      \end{equation}
                    holds for all integers $(n_1,n_2)$.        Note that the zeros of $\widehat{\chi_{[0,1]}}$ are ${\mathbb Z}\setminus\{0\}$. Putting $n_1=1$ and $n_2=0$ yields
                                         $$
                                         \widehat{\chi_{[0,1]}}\left(x_1+\frac{p}{q}\right)=0 \ \mbox{or} \ \widehat{\chi_{[0,1]}}\left(x_2+\frac{r}{s}\right)=0 .
                                         $$
                                         Thus $x_1 =k_1-p/q$ or $x_2 = k_2-r/s$ for some $k_1,k_2\in{\mathbb Z}\setminus\{0\}$.

                                         \medskip
                                         
                                         \noindent{\bf Case (1) $x_1 = k_1-p/q$.} We then put $n_1 = 2$ and $n_2=0$ to conclude that $x_1+2p/q\in{\mathbb Z}\setminus\{0\}$ or $x_2+2r/s\in{\mathbb Z}\setminus\{0\}$. 
                                         
                                         \medskip
                                         
                                         {\it Sub-case (i) $x_1+2p/q\in{\mathbb Z}\setminus\{0\}$.} In this subcase, we have $k_1+p/q\in{\mathbb Z}\setminus\{0\}$, which implies $p/q$ is an integer. Thus, $q=1$. We  now argue that $p> 1$ and hence it must be the second form. Indeed, we know that $x_1\in{\mathbb Z}$. If it happens that $p=1$, we can take $n_1= -x_1$, we will have $x_2-x_1r/s+n_2\in{\mathbb Z}\setminus\{0\}$ for any integer $n_2$. Put $n_2=0$. We obtain that $x_2= \ell+x_1r/s$ for some integer $\ell$. However, we then take $n_2 = -\ell$, we will have $0\in{\mathbb Z}\setminus\{0\}$, which is a contradiction. Thus, $p>1$ and we proved this subcase.
                                         
                                         \medskip
                                         
                                         {\it Sub-case (ii) $x_2+2r/s\in{\mathbb Z}\setminus\{0\} $.} In this subcase, $x_2 = \ell-2r/s$ for some non-zero integer $\ell$. As subcase (i) does not hold, considering $n_1=2$ and any integers $n_2$ we have  $\ell+n_2\frac{q}{p}\in{\mathbb Z}\setminus\{0\}$ for all $n_2$. Putting $n_2=1$, we conclude that $q/p$ is an integer and thus $p=1$. Note that $q>1$ also. Otherwise, we can put $n_2=-\ell$ and we obtain a contradiction. Finally, we have for any integer $n_1,n_2$
                                         \begin{equation}\label{eq6.2}
                                         \widehat{\chi_{[0,1]}}\left(k_1+\frac{n_1-1}{q}\right)=0 \ \mbox{or} \ \widehat{\chi_{[0,1]}}\left(\ell+\frac{(n_1-2)r}{s}+n_2q\right)=0
                                         \end{equation}
                                        
                                        \medskip
                                         
                                        {\it Claim 1: }  $s=1$
                                        
                                        \medskip
                                         
                                         Suppose that $s>1$. If $q=2$, we consider an even number $n_1$, then $k_1+(n_1-1)/2$ cannot be an integer, so $\ell+(n_1-2)r/s+2n_2$ are all non-zero integers.  Thus, putting $n_1=4$, $2r/s$ is an integer, which means $s=2$ is even. we must have $\ell+(n_1/2-1)r+2n_2$ are non-zero integers. In this case, $r$ must be even, otherwise, $2$ and $r$ is relatively prime and we can find $N,n_2$ (and $n_1= 2(N+1)$) such that  $Nr+2n_2=-\ell$, we obtain $\ell+(n_1/2-1)r+2n_2=0$, which is a contradiction. Hence, $r, s$ is not relatively prime. This is a contradiction again. Thus, $s=1$. If $q>2$, then nwe put $n_1=3$, we will have $k_1+2/q$ or $\ell+r/s+qn_2$ is an non-zero integer. Clearly, the first one cannot be as $q>2$, We must have $\ell+r/s+qn_2$ is an integer, which forces $r/s$ is an integer and $s=1$ follows.
                                         
                                         \medskip
                                         
                                         {\it Claim 2:  gcd of $r,q$ must be strictly greater than 1.}
                                         
                                         \medskip
                                         
                                      Now, $x_1,x_2$ are integers and we have $x_1+n_1/q$ or $x_2+n_1r+qn_2$ is a non-zero integer for any integers $n_1,n_2$. Suppose that  the gcd of $r,q$ is one.  Then we can find $n_1,n_2$ such that $n_1r+qn_2 = -x_2$. If $q$ does not divide $x_2$, then $q$ does not divide $n_1$. We will have both $x_1+n_1/q$  and $x_2+n_1r+qn_2$ are not non-zero integers, which is a contradiction. If $q$ divides $x_2$, then we must have $x_2= Nq$. Take $n_1= -x_1q$ and $n_2 = x_1r-N$, then we have all $x_1+n_1/q$ or $x_2+n_1r+qn_2$ equal zero. A contradiction again. Therefore, gcd of $r,q$ must be strictly greater than one.

                                         \medskip
                                         
                                          \noindent{\bf Case (2) $x_2 = k_2-r/s$} we then take $n_1=1$ and $n_2=1$ to conclude that $x_1+p/q$ or $k_2+q/p$ must be a non-zero integer. The first case is back to Case(1), so we are done. In the second case, $p=1$ and for any integer $n_1,n_2$, $x_1+n_1/q$ or $k_2+(n_1-1)r/s+n_2q$ is a non-zero integer. 
                                          
                                          \medskip
                                          
                                          If we take $n_1=0$, then we conclude that $x_1$ or $k_2-r/s+n_2q$ are non-zero integers. If $x_1$ is an integer, then we have the same set of equations as in (\ref{eq6.2}). We can then argue similarly as in Claim 1 and 2 to settle this case. Finally, if $k_2-r/s+n_2q$ is a non-zero integer, we have $s=1$ and  it remains to claim that gcd of $r,q$ is strictly greater than one. Indeed, $x_2$ is an integer. If gcd of $r,q$ is one and $q$ does not divide $x_2$, we can find $N_1,N_2$ such that  $-x_2=N_1r+N_2q$ is a non-zero integer. In this case, $x_1+N_1/q$ is a non-zero integer. concluding that $x_1= k-N_1/q$ for some non-zero integer $k$. We now have $k+(n_1-N_1)/q$ or $x_2+n_1r+n_2q$ is a non-zero integer. However, we can $n_1 = N_1-kq$, and $n_2 = kr+N_2$. Then $x_2+n_1r+n_2q = 0$ and $k+(n_1-N_1)/q = 0$ also, which is a contradiction. Thus, gcd of $r,q$ is strictly greater than one. 

                                                    \end{proof}

\end{document}